\documentclass[a4paper,11pt,english]{article}
\pdfoutput=1

\usepackage[utf8]{inputenc}
\usepackage[T1]{fontenc}
\usepackage{babel}

\usepackage{fullpage}

\usepackage{amsfonts,dsfont,pifont}
\usepackage{amsmath,amsthm,amssymb,amsbsy,mathabx}
\usepackage{upgreek,eufrak}

\usepackage{graphicx,color}
\usepackage[all]{xy}

\usepackage{enumerate}
\usepackage{multirow}

\usepackage{units}
\usepackage{multirow}

\usepackage{url}

\usepackage{cite}

\usepackage[pdftex,colorlinks=true,pagebackref=true
,linkcolor=blue,bookmarks=true,bookmarksopen=true,citecolor=blue]{hyperref}

\usepackage{titlesec}
\usepackage{titletoc}

\usepackage{mathptmx}
\newtheorem{theo}{Theorem}[section]

\newtheorem{lem}{Lemma}[section]
\newtheorem{prop}{Proposition}[section]

\newtheorem{rem}{Remark}[section]

\newtheorem{ass}{Assumption}[section]

\newcommand{\Nset}{\mathbb{N}}
\newcommand{\Rset}{\mathbb{R}}

\newcommand{\E}{\mbox{$\mathbb{E}$}}

\newcommand{\PP}{\mbox{$\mathbb{P}$}}

\newcommand{\CA}{\mathcal{A}}

\newcommand{\scr}[1]{\scriptscriptstyle #1}

\newcommand{\lpref}{\smash{\raisebox{3.5pt}{\!\!\!\begin{tabular}{c}$\hskip-4pt\scriptstyle\longleftarrow$ \\[-7pt]{\rm pref}\end{tabular}\!\!}}}
\newcommand{\petitlpref}{\smash{\raisebox{2.5pt}{\!\!\!\begin{tabular}{c}$\hskip-2pt\scriptscriptstyle\longleftarrow$ \\[-9pt]{$\scriptstyle\hskip 1pt\rm pref$}\end{tabular}\!\!}}}

\newcommand{\ttu}{\mathtt u}
\newcommand{\ttd}{\mathtt d}

\newcommand{\tail}{\mathcal{T}}

\makeatletter
\newcommand{\xRightarrow}[2][]{\ext@arrow 0359\Rightarrowfill@{#1}{#2}}
\makeatother

\newcommand{\longhookrightarrow}{}
\DeclareRobustCommand{\longhookrightarrow}{\lhook\joinrel\relbar\joinrel\rightarrow}

\usepackage[draft]{fixme}
\fxusetheme{color}
\FXRegisterAuthor{a}{aa}{A}
\FXRegisterAuthor{pe}{ape}{P}
\FXRegisterAuthor{yo}{ayo}{Yo}
\FXRegisterAuthor{b}{ab}{B}

\usepackage[blocks]{authblk}

\title{\bf Persistent random walks I : recurrence \emph{versus} transience}

\author[1]{{\bf Peggy Cénac}}
\author[2]{{\bf Basile de Loynes}}
\author[3]{{\bf Arnaud Le Ny}}
\author[1]{{\bf Yoann Offret}}

\affil[1]{Institut de Mathématiques de Bourgogne (IMB) - UMR CNRS 5584\\

Université de Bourgogne, 21000 Dijon, France}

\affil[2]{Institut de Recherche Mathématique Avancée (IRMA) - UMR CNRS 7501\\

Université de Strasbourg, 67084 Strasbourg, France}

\affil[3]{Laboratoire d'Analyse et de Mathématiques Appliquées (LAMA) - UMR CNRS 8050\\

Université Paris Est, 94010 Créteil Cedex, France}

\date{}

\begin{document}

\maketitle

\noindent
\rule{\linewidth}{.5pt}

\vspace{1em}

\noindent

{\small {\bf Abstract}\; We consider a walker on the line that at each step keeps the same direction with a probability which depends on the discrete time already spent in the direction the walker is currently moving. More precisely, the associated left-infinite sequence of jumps is supposed to be a Variable Length Markov Chain (VLMC) built from a probabilized context tree given by a double-infinite comb. These walks with memories of variable length can be seen as generalizations of Directionally Reinforced Random Walks (DRRW) introduced in \cite[Mauldin \& al., Adv. Math., 1996]{Mauldin1996} in the sense that the persistence times are anisotropic. We give a complete characterization of the recurrence and the transience in terms of the probabilities to persist in the same direction or to switch. We point out that the underlying VLMC is not supposed to admit any stationary probability. Actually, the most fruitful situations emerge precisely when there is no such invariant distribution. In that case, the recurrent and the transient property are related to the behaviour of some embedded random walk with an undefined drift so that the asymptotic behaviour depends merely on the asymptotics of the probabilities of change of directions unlike the other case in which the criterion reduces to a drift condition. Finally, taking advantage of this flexibility, we give some (possibly random and lacunar) perturbations results and treat the case of more general probabilized context trees built by grafting subtrees onto the double-infinite comb.}\\

\noindent
{\small {\bf Key words}\; Persistent random walk . Directionally reinforced random walk . Variable length Markov chain . Variable length memory . Probabilized context tree  . Recurrence . Transience . Random walk with undefined mean or drift - Perturbation criterion}\\

\noindent
{\small {\bf Mathematics Subject Classification (2000)}\;   60G50 . 60J15 . 60G17 . 60J05 . 37B20 . 60K35 . 60K37
}\\

\noindent
\rule{\linewidth}{.5pt}

\tableofcontents

\vspace{1em}


\section{Introduction}
\label{intro}

Classical random walks are usually defined from a sequence of independent and identically distributed (\emph{i.i.d.\@}) increments $\{X_k\}_{k\geq 1}$ by
\begin{equation}
\label{def-persist-part}
S_0:=0 \quad \textrm{ and } \quad S_n:=\displaystyle\sum_{k=1}^n X_k \quad \textrm{for all integers} \quad n \geq 1.
\end{equation}

When the jumps are defined as a (finite-order) Markov chain, a short memory in the dynamics of the stochastic paths is introduced and the random walk $\{S_{n}\}_{n\geq 0}$ itself is no longer Markovian. Such a process is called in the literature a \emph{persistent} random walk, a \emph{Goldstein-Kac} random walk or also a \emph{correlated} random walk. Concerning the genesis of the theory, we allude to \cite{Goldstein:51,Kac:74,renshaw81,weissbook94,eckstein00,weiss02} as regards the discrete-time situation but also its connections with the continuous-time telegraph process.

In this paper, we aim at investigating the asymptotic behaviour of one-dimensional random walk in particular random environment, where the latter is partially the trajectory of the walk itself (the past) which acts as a reinforcement. Roughly speaking, we consider a walker that at each step keeps the same direction (or switches) with a probability which directly depends on the time already spent in the direction the walker is currently moving. In order to take into account possibly infinite reinforcements, we need to consider a two-sided process of jumps $\{X_n\}_{n\in\mathbb Z}$ with a variable finite but possibly unbounded memory.

The following probabilistic presentation of the Variable Length Markov Chains (VLMC), initially introduced in \cite{Rissanen}, comes from \cite{ccpp}. Besides, we refer to \cite[pp. 117-134]{INSR:INSR062_1} for an overview on VLMC.


\subsection{VLMC and associated persistent random walks}

Introduce the set $\mathcal{L}= \CA^{-\Nset}$ of left-infinite words on a given alphabet $\CA$ and  consider a complete tree on this alphabet, \emph{i.e.}\@ a tree such that each node has $0$ or $\mathsf{card}(\CA)$ children, whose leaves $\mathcal{C}$ are words (possibly infinite) on $\mathcal A$. To each leaf $c\in\mathcal C$, called a context, is attached a probability distribution $q_{c}$ on $\mathcal A$. Endowed with this probabilistic structure, such a tree is named a probabilized context tree. Different context trees lead to different probabilistic impacts of the past and different dependencies. The model of (very) persistent random walks we consider in this paper corresponds to the double infinite comb which is given in Figure \ref{double-peigne}.

\begin{figure}[!ht]
\begin{center}
\includegraphics[width=0.8\linewidth]{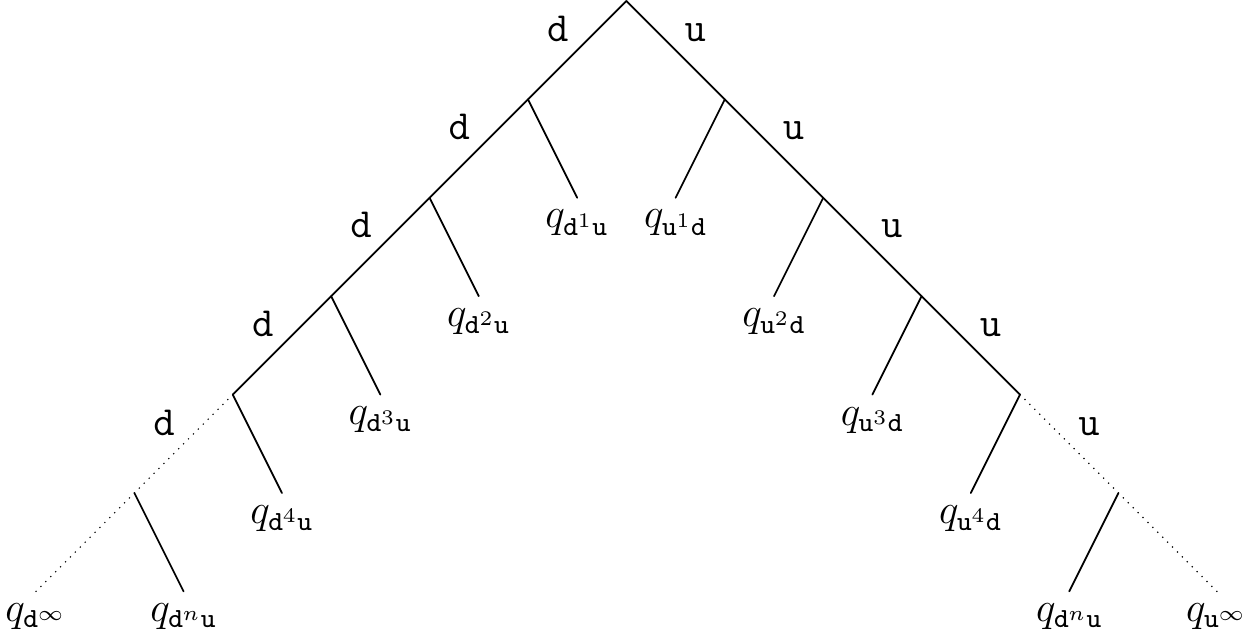}
\end{center}
\caption{\label{double-peigne} Probabilized context tree (double infinite comb).}
\end{figure}

For this particular context tree, the set of leaves $\mathcal{C}$ is defined from a binary alphabet $\mathcal{A}:=\{\ttu,\ttd\}$ consisting of a letter $\ttu$, for moving up, and a letter $\ttd$ for moving {down}. More precisely, it is given by
\begin{equation}\label{leaf}
\mathcal{C}:=\{\ttu^n\ttd : n\ge 1\}\cup \{\ttu^\infty\}\cup\{\ttd^n\ttu : n\ge 1\}\cup \{\ttd^\infty\},
\end{equation}
where $\ttu^n\ttd$ represents the word $\ttu\ldots \ttu\ttd$ composed with
$n$ characters $\ttu$ and one character $\ttd$. Note that the set of leaves contains also two infinite words $\ttu^\infty$ and $\ttd^\infty$. The distributions $q_{c}$ are then Bernoulli distributions and we set, for any $\ell\in\{\ttu,\ttd\}$ and $n\geq 1$,
\begin{equation}
\label{eq:double-peigne}
q_{\ttu^n \ttd}(\ttu) =  1-q_{\ttu^n\ttd}(\ttd) =: 1-\alpha^\ttu_n,\quad
q_{\ttd^n\ttu}(\ttd) = 1-q_{\ttd^n\ttu}(\ttu)=:1-\alpha^\ttd_n\quad \mbox{and}\quad q_{\ell^\infty}(\ell)=:1-\alpha_{\infty}^{\ell}.
\end{equation}
It turns out from equalities (\ref{changedir}) below that $\alpha_{n}^\ell$ given in (\ref{eq:double-peigne}) stands for the probability of changing letter after a run of length $n$ (possibly infinite) of letters $\ell$.

For a general context tree and any left-infinite word $U\in\mathcal L$, the prefix $\lpref(U)\in\mathcal C$ is defined as the shortest suffix of $U$, read from right to left, appearing as a leaf of the context tree. In symbols, it is given for any $\ell\in\{\ttu,\ttd\}$ and $n\geq 1$  by
\begin{equation}
 \lpref (\ldots \ttd\ttu^{n})=\ttu^n\ttd,\quad \lpref (\ldots \ttu\ttd^{n})=\ttd^n\ttu\quad\mbox{and}\quad \lpref (\ell^{\infty})=\ell^{\infty}.
\end{equation}
Then the associated VLMC, entirely determined by $\{q_{c} : c \in \mathcal{C}\}$ and $U_0\in\mathcal L$, is the $\mathcal{L}$-valued Markov chain $\{U_n\}_{n\geq 0}$  given, for any $\ell\in\mathcal A$ and $n\geq 0$, by the transitions
\begin{equation}
 \label{eq:def:VLMC}
 \PP(U_{n+1} = U_n\ell| U_n)=q_{\petitlpref (U_n)}(\ell).
 \end{equation}
Since $\mathcal L$ is naturally endowed with a one-sided space-shift corresponding to the left inverse of the usual time-shift, the whole past can be recovered from the state of $U_n$ for any $n \geq 0$. Thus we can introduce $X_n\in\mathcal A$ for any $n\in\mathbb Z$  as the rightmost letter of $U_n$. In particular, we can write
\begin{equation}
U_{n}=\cdots X_{n-1}X_{n}.
\end{equation}
Supposing the context tree is infinite,  the so called letter process $\{X_n\}_{n\in\mathbb Z}$ is not a finite-order Markov chain in general. Furthermore, given an embedding $\mathcal A\longhookrightarrow  G$ of the alphabet in an additive group $G$, the resulting random walk $\{S_n\}_{n\geq 0}$ defined in \eqref{def-persist-part} is no longer Markovian and somehow very persistent.

In the sequel, we make the implicit coding $\{\ttd,\ttu\}\simeq \{-1,1\}\subset\mathbb R$ ($\ttd$  for a descent and $\ttu$ for a rise) so that the letter process $\{X_{n}\}_{n\in\mathbb Z}$ represents the jumps in $\{-1,1\}$  of the persistent random walk $S$ taking its values in $\mathbb Z$. In other words, the persistent random walk is defined by the transitions probabilities of changing directions, that is, for all $n \geq 1$, $m \geq 0$
\begin{equation}\label{changedir}
	\mathbb P(S_{m+1}=S_{m}+ 1 \,|\, U_{m}=\cdots {\ttu} \ttd^{n})=\alpha^{\ttd}_{n}\quad\mbox{and}\quad 	\mathbb P(S_{m+1}=S_{m}- 1 \,|\, U_{m}=\cdots {\ttd} \ttu^{n})=\alpha^{\ttu}_{n}.
\end{equation}


\subsection{Overview of the results}

In this paper, necessary and sufficient conditions on the transitions probabilities $\alpha_{n}^{\ell}$ for the corresponding persistent random walk $S$ to be recurrent or transient are investigated.

Under some moment conditions on the persistence times (to be defined) whose distributions depend only on the $\alpha_n^\ell$'s, a Strong Law of Large Numbers (SLLN) as well as a Central Limit Theorem (CLT) are stated in \cite{peggy}.  These conditions imply that the underlying VLMC admits a unique stationary probability distribution.

In the following, we extend these results by providing a complete characterization of the recurrence and the transience. We also slightly relax the assumptions for the SLLN. These results are robust in the sense that they do not rely neither on the existence of any stationary probability nor on any reversibility property. A summary of the different situations is given in Table \ref{tableau}. The generalization of the CLT in \cite{peggy}, when we do not assume the square integrability of the running times, is a work being drafted for a forthcoming companion paper.

Basically,  when the random elapsed times between two changes of directions (also termed the persistence times) $\tau^{\ttu}$ and $\tau^{\ttd}$ are integrable (or at least one of them) the criterion for the recurrence reduces to a classical null drift condition. In the remaining case, the recurrence requires the distribution tails of the persistence times to be comparable. Thus, in the former case, the criterion involves the parameters $\alpha_n^\ell$ globally whereas it only depends on their asymptotics in the latter case. It follows, in the undefined drift context, that we can slightly perturb the symmetric configuration while remaining recurrent contrary to the well-defined drift case for which a perturbation of exactly one transition can lead to a transient behaviour. In fact, in the undefined drift case, the persistent random walk may stay recurrent or transient as long as the perturbation remains asymptotically controlled.

The proofs of these results rely on the study of the skeleton (classical) random walk
\begin{equation}
\{M_{n}\}_{n\geq 0}:= \{S_{T_{n}}\}_{n\geq 0},
\end{equation}
where $\{T_{n}\}_{n\geq 0}$ is the (classical) random walk of up-down breaking times. Then we use classical results, especially the structure theorem in \cite{Erickson2} of Erickson on random walks with undefined mean. These random walks together with the length of runs, also called the times of change of direction, are illustrated in Figure \ref{marche} at the beginning of Section \ref{def-marche-persist}.


\subsection{Related results}

First, some optimal stopping rules have been obtained in \cite{Allaart2001,Allaart2008} for the walks considered. Secondly, recurrence and transience as well as scaling limits have been widely investigated for  correlated random walks, that is, in our framework, persistent random walks with a probabilized context tree of finite depth (in particular, the increment process is a finite-order Markov chain). For instance, regarding persistent random walks in random environments (where the latter are the transition probabilities to change direction) CLT are proved in \cite{Toth84,Toth86}. Besides, recurrence and transience have been studied in \cite{renshaw83,Lenci} for correlated random walks in dimension two.

Closely related to our model, Directionally Reinforced Random Walks (DRRW) has been introduced in \cite{Mauldin1996} to model ocean surface wave fields. Those are nearest neighborhood random walks on $\mathbb Z^{d}$ keeping their directions during random times $\tau$, independently and identically drawn after every change of directions, themselves independently and uniformly chosen among the other ones. In dimension one, our model generalizes these random walks since asymmetrical transition probabilities $(\alpha_{n}^{\ttu})$ and $(\alpha_{n}^{\ttd})$ lead in general to running times $\tau^{\ttu}$ and $\tau^{\ttd}$ with distinct distributions.

Due to their symmetry, the recurrence criterion of DRRW in dimension one takes the simple form given in \cite[Theorem 3.1., p. 244]{Mauldin1996} and obviously we retrieve this particular result in our more general situation (see Proposition \ref{crit-rec} and Theorem \ref{undefinedrt}). Under some significant moment conditions on the running time, it is stated in \cite[Theorem 3.3. and Theorem 3.4., p. 245]{Mauldin1996} that these random walks are recurrent in $\mathbb Z^{2}$, when the waiting time between changes of direction is square integrable, and transient in $\mathbb Z^{3}$ under the weaker assumption of a finite expectation. In higher dimension, it is shown that it is always transient. In dimension three the latter result has been recently  improved in \cite[Theorem 2., p. 682]{Rainer2007} by removing the integrability condition. Also, the assertion in \cite{Mauldin1996} that the DRRW is transient when its embedded random walk of successive locations of change into the first direction is transient has been partially invalided in \cite[Theorem 4., p. 684]{Rainer2007}. Thus, even in the symmetric situation, the characterization of recurrence or transience is a difficult task. The case of an anisotropic persistent random walks built from VLMC in higher dimension is a work in progress and this paper is somehow a first step.

Besides, as regards the scaling limits of DRRW, we refer to \cite{Scalingdirectionaly1998,Rastegar:2012,Scalingdirectionaly2014} where are revealed diffusive and super-diffusive behaviours. We expect in a forthcoming paper to extend these results to the asymmetric situation and fill some gaps left open.

Finally, these walks are also somehow very similar to some continuous-time random motion, also called random flights, for which the changes of directions are Poisson random clocks. They have been extensively considered as generalizations of the Goldstein-Kac telegraph process. A sample of this field can be found for instance in \cite{Kolesnik:2005,Orsingher2007,Kolesnik:2008}.


\subsection{Outline of the article}

The paper is organized as follows. Section~\ref{def-marche-persist} is devoted to the presentation of the general framework including the main assumptions and notations summarized in Figure \ref{marche}. Section~\ref{Rec-et-Trans} is focused on the recurrence and transience. It is first shown that the recurrent or transient behaviour of $S$ can be deduced from the oscillating or drifting behaviour of the skeleton classical random walk $M$. From this follows an almost sure comparison lemma involving stochastic domination in the context of couplings. The main result of this paper requires to discriminate two situations depending on whether the persistent random walk admits an almost sure drift or not. The bulk of the work in the latter case consists of giving a characterization in a form as simple as possible by applying the results in \cite{Erickson2}. To this end we need to reduce the problem to the study of a derived random walk which can be viewed as a randomized version of the embedded random walk $M$. Finally, in Section \ref{perturbations} we give some perturbation results in the undefined drift case and also an {\it a.s.\@} comparison result weakening the assumptions of the former one providing all the needed informations in the specific undefined drift context.

\section{Settings and assumptions}
\label{def-marche-persist}

\setcounter{equation}{0}

\begin{figure}[!ht]
\begin{center}
\includegraphics[width=0.80\linewidth]{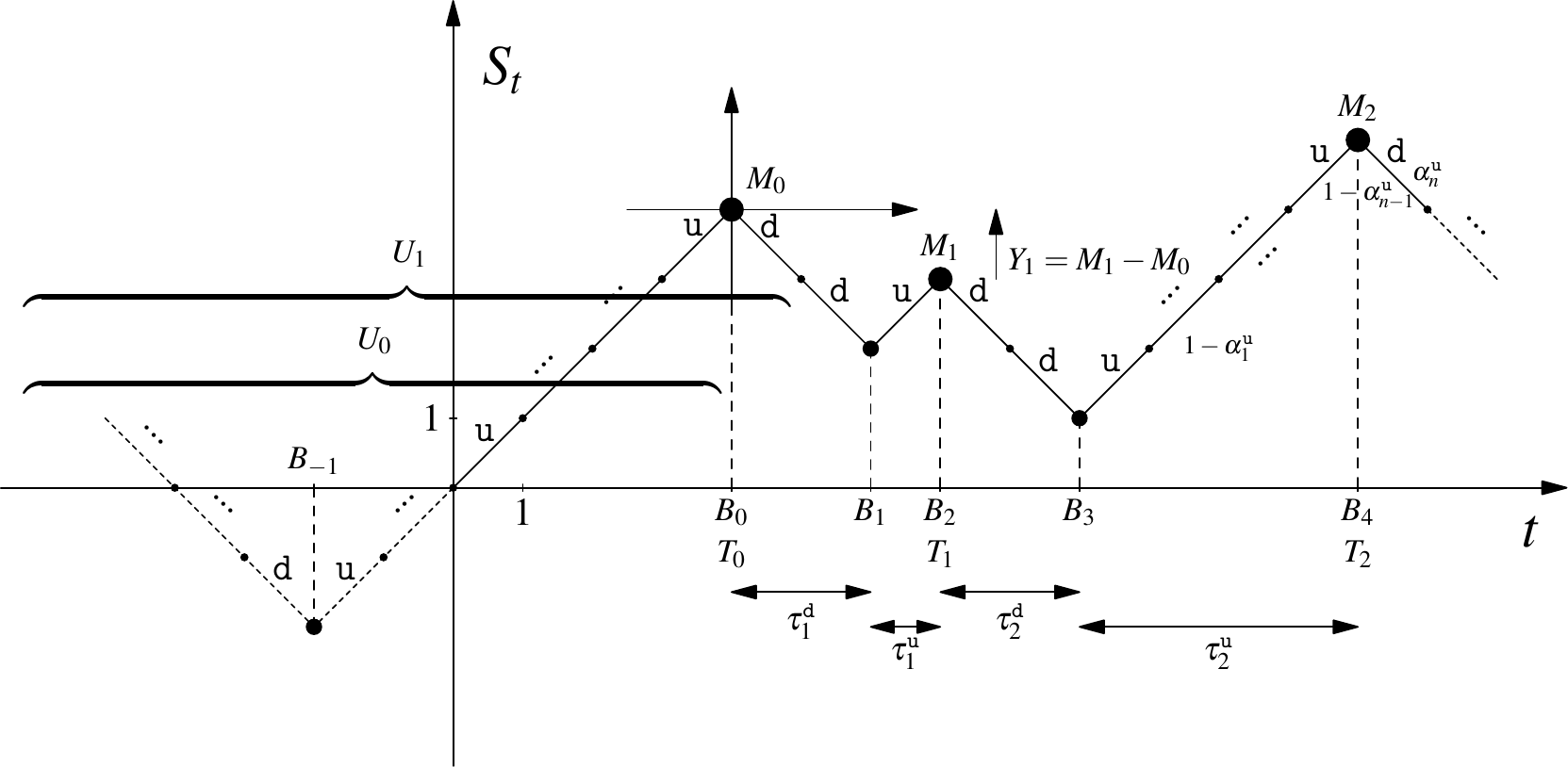}
\end{center}
\caption{Persistent random walk}
\label{marche}
\end{figure}

Foremost, we refer to Figure \ref{marche} that illustrates our notations and assumptions by a realization of a linear interpolation of our persistent random walk $\{S_n\}_{n\geq 0}$, built from a double infinite comb given in Figure \ref{double-peigne} with probabilities of changing directions $(\alpha_{n}^{\ttu})$ and $(\alpha_{n}^{\ttd})$ as in (\ref{eq:double-peigne}).


\subsection{Renewal hypothesis}

In order to avoid the trivial situations in which the persistent random walk can stay almost surely frozen in one of the two directions,  we make the assumption below. It roughly means that the underlying VLMC denoted by $U$ given in (\ref{eq:def:VLMC}) has a renewal property. We refer to \cite{Comets,Gallo} for more refinement about regenerative schemes and related topics.

\begin{ass}[renewal set and finiteness of the length of runs]\label{A1} For any initial distribution $\mu$ on $\mathcal L$,
\begin{equation}\label{a0}
\mathbb P_{\mu}\left(\lpref(U_{n})=\ttd\ttu\quad i.o.\right)
=
\mathbb P_{\mu}\left((X_{n-1},X_{n})=(\ttu,\ttd) \quad i.o.\right)=1,
\end{equation}
where the abbreviation {\it i.o.\@}  means that the events depending on $n$ occur for infinitely many $n$. It turns out that this hypothesis is equivalent to one of the following statements:
\begin{enumerate}
\item For any $\ell\in\{\ttu,\ttd\}$ and $r\geq 1$, $\alpha_{\infty}^{\ell}\neq 0$ and
\begin{equation}\label{a1}
\prod_{k=r}^{\infty}(1-\alpha_{k}^\ell)=0.
\end{equation}
\item For any $\ell\in\{\ttu,\ttd\}$ and  $r\geq 1$, $\alpha_{\infty}^{\ell}\neq 0$ and, either there exists $n\geq r$ such that $\alpha_{n}^{\ell}=1$, or
\begin{equation}\label{a1b}
\sum_{k=r}^{\infty} \alpha_{k}^{\ell}=\infty.
\end{equation}
\end{enumerate}
Furthermore, since between two up-down events $(X_{n-1},X_{n})=(\ttu,\ttd)$ there is at least one down-up event $(X_{n-1},X_{n})=(\ttd,\ttu)$ and reciprocally, assumption (\ref{a0}) can be alternatively stated switching $\ttu$ and $\ttd$.
\end{ass}

This assumption disallows a too strong reinforcement, that is a too fast decreasing rate for the probabilities of change of directions. Sequences of transition satisfying this assumption are said to be {admissible}. Below are given typical examples for which the assumption holds or fails.


\subsubsection*{Examples of admissible or inadmissible sequences}

{\it
Denoting for any integer $p\geq 0$, the $p$-fold composition of the logarithm function by
\begin{equation}
\log_{[p]}:=\log\circ\cdots\circ\log,
\end{equation}
it is obvious that (\ref{a1b}) holds for instance when
\begin{equation}\label{ex2moins}
\frac{1}{n\log(n)\cdots\log_{[p]}(n)}\underset{n\to\infty}{=}{\mathcal O}(\alpha_{n}^{\ell}),
\end{equation}
where $\mathcal O$ stands for the big O notation. In contrast, it fails when there exists $\varepsilon>0$ such that
\begin{equation}
\alpha_{n}^{\ell}\underset{n\to\infty}{=}
{\mathcal O} \left(\frac{1}{n\log(n)\cdots\log_{[p-1]}(n)(\log_{[p]}(n))^{1+\varepsilon}}\right).
\end{equation}}


\subsection{Persistence times and embedded random walk}

Thereafter, under Assumption \ref{A1}, we can consider the sequence of almost surely finite breaking times $(B_n)_{n\geq 0}$ (see Figure \ref{marche}) defined inductively for all $n\geq 0$ by
\begin{equation}
B_0=\inf\left\{k\geq 0 : X_{k}\neq X_{k+1}\right\}\quad\mbox{and}\quad
B_{n+1}=\inf\left\{k>B_{n} : X_{k}\neq X_{k+1}\right\}.
\end{equation}
For the sake of simplicity, throughout this paper, we deal implicitly with the conditional probability
\begin{equation}\label{condition}
\mathbb P(\;\cdot\;\mid \lpref (U_{1})=\ttd\ttu)=\mathbb P(\;\cdot\;\mid (X_{0},X_{1})=(\ttu,\ttd)).
\end{equation}
In particular, $B_{0}=0$ with probability one. In other words, we suppose  that the initial time is a so called up-down breaking time. Thanks to Assumption \ref{A1} and to the renewal properties of the chosen variable length Markov chain $U$, the latter condition can be done without loss of generality and has no fundamental importance in the long time behaviour of $S$ . Furthermore, the length of rises $(\tau^\ttu_{n})$ and of descents $(\tau^\ttd_{n})$ are then defined for all $n\geq 1$ by
\begin{equation}
\label{tau}
\tau_n^{\ttd}:=B_{2n-1}-B_{2n-2}
\quad\mbox{and}\quad
\tau_n^{\ttu} :=B_{2n}-B_{2n-1}.
\end{equation}
Due to \cite[Proposition 2.6.]{peggy} for instance, $(\tau_n^\ttd)$ and $(\tau_n^\ttu)$ are independent sequences of {\it i.i.d.\@} random variables. Besides, an easy computation leads to their distribution tails and expectations, given and denoted for any $\ell\in\{\ttu,\ttd\}$ and all integers $n\geq 1$ by
\begin{equation}
\label{def-tail}
\mathcal T_{\ell}(n):=\PP(\tau^{\ell}_{1} \geq n)=\prod_{k=1}^{n-1}(1-\alpha_{k}^\ell)\quad\mbox{and}\quad \Theta_\ell(\infty):=\E[\tau^{\ell}_{1}]=\sum_{n=1}^{\infty}\prod_{k=1}^{n-1}(1-\alpha_{k}^\ell).
\end{equation}

At this stage, we exclude for simplicity the situation of almost surely constant length of runs which trivializes the analysis of the underlying persistent random walk. Besides, note that the persistent random walk $S$ can be equivalently defined either via the distribution tails $\tail_\ell$ or the probabilities $(\alpha_n^\ell)$ with $\ell\in\{\ttd,\ttu\}$. Thus, depending on the context, we will choose the more suitable description of the parameters of the model.

In order to deal with a more tractable random walk built with possibly unbounded but  \emph{i.i.d.\@} increments, we consider the underlying skeleton random walk $\{M_{n}\}_{n\geq 0}$ associated with the even breaking times random walk $\{T_{n}\}_{n\geq 0}$ (up-down breaking times) defined for all $n\geq 0$ by
\begin{equation}
\label{marche-paire}
T_{n}:=\sum_{k=1}^{n}(\tau_{k}^{\ttd}+\tau_{k}^{\ttu})\quad\mbox{and}\quad
M_{n}:=S_{T_{n}}=\sum_{k=1}^{n}(\tau^{\ttu}_{k}-\tau^{\ttd}_{k}).
\end{equation}
The so called (almost sure) drift $\mathbf d_{\scr T}$ of the increasing random walk $T$ always exists whereas that of $M$, denoted by $\mathbf d_{\scr M}$, is well-defined whenever one of the expectations $\Theta_{\ttu}(\infty)$ or $\Theta_{\ttd}(\infty)$ defined in the right-hand side of (\ref{def-tail}) is finite. They are  given by
\begin{equation}
\label{drift-def1}
{\mathbf d_{\scr T}}=\Theta_{\ttd}(\infty)+\Theta_{\ttu}(\infty)\in[0,\infty]\quad\mbox{and}\quad {\mathbf d_{\scr M}}:=\Theta_{\ttu}(\infty)-\Theta_{\ttd}(\infty)\in\Rset\cup\{\pm \infty\}.
\end{equation}
Furthermore, if the drift of $M$ is well-defined, we can set (extended by continuity whenever necessary)
\begin{equation}
\label{drift-def2}
{\mathbf d}_{\scr S}:=\frac{\Theta_{\ttu}(\infty)-\Theta_{\ttd}(\infty)}{\Theta_{\ttu}(\infty)+\Theta_{\ttd}(\infty)}\in [-1,1].
\end{equation}
In regards to the convergence (\ref{LFGN1}), the latter quantity is naturally termed the (almost sure) drift of $S$. For the record, we give some relevant examples for which the mean of the length of run is finite or not.

In the following, two non-negative sequences $(a_n)$ and $(b_n)$ are said to be of the same order, and we shall write it $a_n  \asymp  b_n$, when there exists a positive constant $c$ such that for  $n$ sufficiently large,
\begin{equation}
c^{-1} a_n \leq b_n \leq c a_n.
\end{equation}


\subsubsection*{Typical examples of running time means}

{\it
Simple calculations ensure that the expectation $\Theta_{\ell}(\infty)$ is finite whenever there exists  $p\geq 0$ and $\varepsilon>0$ such that either there exists $n\geq 1$ with $\alpha_{n}^{\ell}=1$ or, for $n$ sufficiently large,
\begin{equation}\label{ex1}
\alpha_{n}^{\ell}\geq \frac{1}{n}+\frac{1}{n\log(n)}+\cdots+\frac{1}{n\log(n)\cdots\log_{[p-1]}(n)}+\frac{1+\varepsilon}{n\log(n)\cdots\log_{[p]}(n)}.
\end{equation}
On the contrary, $\Theta_{\ell}(\infty)$ is infinite whenever there exists $p\geq 0$ such that, for all $n\geq 1$ we have $\alpha_{n}^{\ell}\neq 1$ and for $n$ large enough,
\begin{equation}\label{ex2}
\alpha_{n}^{\ell}\leq \frac{1}{n}+\frac{1}{n\log(n)}+\cdots+\frac{1}{n\log(n)\cdots\log_{[p]}(n)}.
\end{equation}
The proofs of these claims follow from the computation of the asymptotics of the distribution tails. More precisely, when the transitions belongs to $[0,1)$ and are given, for positive parameters $\lambda_{0},\lambda_{1},\cdots,\lambda_{p}$ and large $n$, by
\begin{equation}\label{ex4}
\alpha_{n}^{\ell}:=\frac{\lambda_0}{n}+\frac{\lambda_{1}}{n\log(n)}+\cdots+\frac{\lambda_{p}}{n\log(n)\cdots\log_{[p]}(n)},
\end{equation}
then,
\begin{equation}\label{tailasymp}
\tail_{\ell}(n) \asymp \frac{1}{n^{\lambda_{0}}(\log(n))^{\lambda_{1}}\cdots(\log_{[p]}(n))^{\lambda_{p}}},
\end{equation}
and the claims follow.
}

\section{Recurrence and transience}
\label{Rec-et-Trans}

\setcounter{equation}{0}

First recall that a stochastic process $\{M_n\}_{n\geq 0}$ on the grid  $\mathbb Z$ is said to be recurrent if for any $x\in\mathbb Z$,
\begin{equation}
\sup\{n\geq 0 : M_n=x\}=\infty\quad a.s.,
\end{equation}
and it is said to be transient (respectively transient to $\infty$, transient to $-\infty$) if
\begin{equation}
\lim_{n\to\infty} |M_n|=\infty\quad\left(\mbox{resp.}\quad
\lim_{n\to\infty} M_n=\infty,\quad\lim_{t\to\infty} M_n=-\infty\right)\quad a.s..
\end{equation}

Below, it is shown that the recurrence or the transience property of the persistent random walk $S$ is completely determined by the oscillating or drifting behaviour of the underlying skeleton random walk $M$ for which suitable criteria are available. Also, a comparison theorem, involving only the distribution tails of the length of runs, is given.

\subsection{Equivalent criteria and comparison lemma}

The following structure theorem, stated in \cite[Theorem 1., Chap. XII and Theorem 4., Chap. VI]{Feller} for instance, describes the long time behaviour of  a one-dimensional random walk.
\begin{theo}[type of a one-dimensional random walk]\label{RW}
Any random walk $M$ on $\mathbb R$ which is not almost surely constant satisfies either
\begin{equation}
\limsup_{n\to\infty} M_{n}=\infty\quad\mbox{and}\quad\liminf_{n\to\infty} M_{n}=-\infty\quad a.s.\quad(\mbox{said to be oscillating}),
\end{equation}
or
\begin{equation}
\lim_{n\to\infty} M_{n}=\infty\quad(\mbox{resp.}\; -\infty)\quad a.s.\quad(\mbox{said to be drifting to}\;\pm\infty).
\end{equation}
Moreover, when the drift of $M$ denoted by ${\mathbf d}_{\scr M}$ is well-defined, then
$M$ is oscillating if and only if ${\mathbf d}_{\scr M}=0$ whereas $M$ is drifting to $\infty$ (resp. $-\infty$) if and only if $\mathbf d_{\scr M}>0$ (resp. $\mathbf d_{\scr M}<0$). In any case,
\begin{equation}
\lim_{n\to\infty}  \frac{M_{n}}{n}={\mathbf d}_{\scr M}\quad a.s..
\end{equation}
\end{theo}

Our strategy to study recurrence \emph{versus} transience consists in reducing the determination of the type of the persistent random walk $S$ defined in the previous section by studying some properties of the underlying embedded random walk $M$ associated with the up-down breaking times given in (\ref{marche-paire}) and illustrated in Figure \ref{marche}. This is made clear by the following lemma.

\begin{lem}\label{eqrec}
The persistent random walk $S$ is either recurrent or transient according as the type (in the sense of Theorem \ref{RW}) of the classical random walk $M$ of even breaking times. More precisely, one has:
\begin{enumerate}[a)]
\item[1)] $S$ is recurrent if and only if $M$ is oscillating.
\item[2)] $S$ is transient to $\infty$ (resp. $-\infty$) if and only if $M$ is drifting to $\infty$ (resp. $-\infty$).
\end{enumerate}
\end{lem}

\begin{proof}[Proof of Lemma~\ref{eqrec}] First, when $M$ is oscillating, $S$ is recurrent. Next if $M$ is drifting to $-\infty$, then $S$ is transient to $-\infty$ since the trajectory of $S$ is always under the broken line formed by the $M_n$'s. Finally, from Theorem \ref{RW}, the oscillating and drifting to $\pm\infty$  behaviour form, up to a null set, a  partition of the universe. Therefore, it only remains to prove that if $M$ is drifting to $\infty$, then $S$ is transient to $\infty$. It is worth to note that we assume the initial time to be an up-down breaking time as in Figure~\ref{marche} so that the geometric argument considered above does not apply straightforwardly. Nonetheless, the expected assertion follows by remarking that, up to an independent random variable, the skeleton random walk at odd breaking times (down-up breaking times) is equal in distribution to $M$ which ends the proof of the lemma.
\end{proof}

Let us end this part with a comparison lemma helpful to study some perturbed persistent random walks (see Section \ref{perturbations}) but also to prove the extended SLLN in Proposition \ref{crit-rec}. It means that if the rises of $S$ are stochastically smaller than those of $\widetilde S$ and the opposite for the descents, then $S$ is stochastically smaller than $\widetilde S$.

\begin{lem}[comparison lemma]\label{comp}
Let $S$ and $\widetilde S$ be two persistent random walks  such that the associated distribution tails of their length of runs satisfy for all $n\geq 1$,
\begin{equation}\label{comp1}
\tail_{\ttu}(n) \leq \widetilde {\tail}_\ttu(n) \quad\mbox{and}\quad\tail_{\ttd}(n) \geq\widetilde{\tail}_{\ttd}(n).
\end{equation}
Then there exists a coupling, still denoted by $(S,\widetilde S)$ up to a slight abuse, such that for all $n \geq 1$,
\begin{equation}\label{comp2}
S_n\leq \widetilde S_n\quad a.s..
\end{equation}
\end{lem}
This lemma can equivalently be stated in terms of the transition probabilities of change of directions, denoted respectively by $(\alpha^{\ell}_n)$ and $(\widetilde \alpha^{\ell}_n)$ for any $\ell\in\mathcal \{\ttd,\ttu\}$, with the same conclusions, by considering instead of \eqref{comp1} the equivalent hypothesis requiring that for all $n\geq 1$,
\begin{equation}\label{comp3}
{\widetilde \alpha}^{\ttu}_{n}\leq \alpha_{n}^{\ttu}\quad\mbox{and}\quad\widetilde \alpha^{\ttd}_{n}\geq \alpha_{n}^{\ttd}.
\end{equation}

\begin{proof}[Proof of Lemma \ref{comp}]
Let $(\tau_{n}^{\ell})$ and $(\widetilde{\tau}_{n}^{\ell})$ be the associated lengths of runs  and  $G_{\ell}$ and $\widetilde{G}_{\ell}$ be the left continuous inverse of their cumulative distribution functions. Then inequalities in (\ref{comp1}) yield that for all $x\in [0,1]$,
\begin{equation*}
G_{\ttu}(x)\leq \widetilde G_{\ttu}(x)\quad\mbox{and}\quad
G_{\ttd}(x)\geq \widetilde G_{\ttd}(x).
\end{equation*}
Then we can construct a coupling (see for instance the book \cite[Chap. 1.3.]{Thorisson}) of the lengths of runs such that, with probability one, for all $n\geq 1$,
\begin{equation*}
\tau_{n}^{\ttu}\leq \widetilde \tau_{n}^{\,\ttu}\quad\mbox{and}\quad \tau_{n}^{\ttd}\geq \widetilde \tau_{n}^{\,\ttd}.
\end{equation*}
To be more specific, considering two independent sequences $(V^{\ell}_{n})$ of uniform random variables on $[0,1]$, we can set
\begin{equation*}
\tau_{n}^{\ell}:=G_{\ell}(V_{n}^{\ell})\quad\mbox{and}\quad \tilde{\tau}_{n}^{\ell}:=\widetilde G_{\ell}(V_{n}^{\ell}).
\end{equation*}
Consequently, there exists a coupling of the persistent random walks $S$ and $\widetilde S$ satisfying inequality  (\ref{comp2}) since they are entirely determined by these lengths of runs.
\end{proof}

With respect to the considerations above, it seems natural to distinguish two cases providing whether one of the mean length of runs between $\Theta_\ttu(\infty)$ and $\Theta_\ttd(\infty)$ given in (\ref{def-tail}) is finite or both are infinite. The former case correspond to the situation in which the drift of $M$ is well-defined and is considered in the next section. The latter case, when the definition of the drift in \eqref{drift-def2} is meaningless, is considered apart in Section \ref{und-drift}.


\subsection{Well-defined Drift case}
\label{wd-drift}

In this part, assume that the drift is well defined, that is $\Theta_\ttu(\infty)$ or $\Theta_\ttd(\infty)$ is finite so that $\mathbf d_{\scr S}$ given in \eqref{drift-def2} is well-defined. We will highlight a Strong Law of Large Number (SLLN) for the persistent random walk and we shall prove a null drift recurrence criterion similarly to the classical context of random walks with integrable jumps.

\begin{prop}[recurence criterium and SLLN] \label{crit-rec}
The persistent random walk $S$  is recurrent if and only if $\mathbf d_{\scr S}=0$ and transient otherwise. Furthermore, one has
\begin{equation}\label{LFGN1}
\lim_{t\to\infty}\frac{S_{n}}{n}=\mathbf d_{\scr S}\in[-1,1]\quad a.s..
\end{equation}
\end{prop}

\begin{proof}
First remark that, in this setting, the recurrence criterion is a straightforward consequence of Theorem \ref{RW} and Lemma \ref{eqrec}. Besides, the law of large numbers (\ref{LFGN1}) when $\Theta_{\ttu}(\infty)$ and $\Theta_{\ttd}(\infty)$ are both finite is already proved in \cite[Proposition 4.5, p. 33]{peggy} under the assumption that $\mathbf d_{\scr S}\in(-1,1)$. Then by symmetry it only  remains to consider the situation with $\Theta_{\ttu}(\infty)=\infty$ and $\Theta_{\ttd}<\infty$ (and thus
$\mathbf d_{\scr S}=1$). Note that it is sufficient to prove
the minoration in (\ref{LFGN1}) since $S_{n}\leq n$ for all $n\geq 0$. To this end, we shall construct for any $0<\varepsilon<1$ a persistent random walk  $S^{\,\varepsilon}$ such that for all
$n\geq 0$,
\begin{equation}\label{minor}
S_n^{\,\varepsilon}\leq S_n\quad\mbox{and}\quad \liminf_{t\to\infty}\frac{S^{ \,\varepsilon}_{n}}{n}\geq 1-\varepsilon\quad a.s..
\end{equation}
More specifically, introduce for any $\ell\in\{\ttd,\ttu\}$ the truncated mean of runs defined for all $m\geq 1$ by
\begin{equation}\label{truncatedmeaninit}
\Theta_{\ell}(m):=\sum_{n=1}^{m}\prod_{k=1}^{n-1}(1-\alpha_{k}^\ell),
\end{equation}
and choose $N\geq 1$ sufficiently large so that
\begin{equation}\label{approx}
\frac{\Theta_{\ttu}(N)-\Theta_{\ttd}(\infty)}{\Theta_{\ttu}(N)+\Theta_{\ttd}(\infty)}\geq
1-\varepsilon.
\end{equation}
Then, we can consider a persistent random walk $S^{\,\varepsilon}$ associated with the transitions
\begin{equation*}
\label{suite-admissibles-recurrence}
\alpha_{n}^{\ttu,\varepsilon}:=
\left\{
\begin{array}{lll}
\alpha_{n}^{\ttu},&\mbox{when} & 1\leq n\leq N-1,\\
1,&\mbox{when}& n\geq N,\\
\end{array}
\right.
\quad\mbox{and}\quad
\alpha_n^{\ttd,\varepsilon}:= \alpha_n^{\ttd}.
\end{equation*}
Finally, noting that the drift of $S^{\,\varepsilon}$ is nothing but the left-hand side of (\ref{approx}), the latter SLLN together with the comparison Lemma \ref{comp} lead to (\ref{minor}) and end the proof.
\end{proof}


\subsection{Undefined drift case}
\label{und-drift}

In this section we consider the remaining case in which both $\Theta_\ttu(\infty)$ and $\Theta_\ttd(\infty)$ are infinite.  In this case, the information given by the expectation of one increment of $M$ is no longer sufficient to discriminate between transience and recurrence.

In fact, following Erickson \cite[Theorem 2., p. 372]{Erickson2}, the oscillating or drifting behaviour of the skeleton random walk $M$ is characterized through the cumulative distribution function of its increments $(Y_{n}):=(\tau^{\ttu}_{n}-\tau^{\ttd}_{n})$, especially if the mean is undefined. Roughly speaking, the criterion of Erickson together with the lemma \ref{eqrec} imply that the persistent random walk $S$ is recurrent if the distribution tails of the positive and negative parts of an increment are comparable, transient otherwise.

However, Erickson's criterion does not suit to our context since the distribution of an increment is not explicitly given by the parameters of the model, but merely by the convolution of two \emph{a priori} known distributions. More precisely, this criterion requires to settle whether the quantities
\begin{equation}\label{erick1}
J_{+}:=\sum_{n=1}^{\infty}\frac{n\mathbb P(Y_{1}=n)}{\sum_{k=1}^{n} \mathbb P(Y_{1}\leq -k)}\quad\mbox{and}\quad J_{-}:=\sum_{n=1}^{\infty}\frac{n\mathbb P(Y_{1}=-n)}{\sum_{k=1}^{n} \mathbb P(Y_{1}\geq  k)},
\end{equation}
are finite or infinite which is clearly not convenient in concrete cases. To circumvent these difficulties, we consider a sequence $(\xi_n)$ of non-degenerate {\it i.i.d.}\@ Bernoulli random variables with parameter $p\in(0,1)$, independent of the sequences of length of runs $(\tau_n^\ttu)$ and $(\tau_n^\ttd)$. Then we introduce the following classical random walk defined for all $n\geq 0$ by
\begin{equation}\label{Mxi-rw}
M^{\scr \xi}_n := \sum_{k=1}^n Y_k^{\scr \xi},\quad  \mbox{with}\quad Y_k^{\scr \xi}:=\xi_k \tau_k^\ttu - (1-\xi_k) \tau_k^\ttd.
\end{equation}
In fact, the original random walk $M$ is built from a strict alternation of rises and descents. The random walk $M^{\scr \xi}$ can be seen as a randomized version of $M$ in the sense that the choice of a descent or a rise in the alternation is determined by flipping a coin. It turns out that the randomly modified random walk $M^{\scr \xi}$ and the embedded random walk $M$ share the same behaviour in the sense of Theorem \ref{RW} above.

When the drift is well-defined this result is only true in the symmetric situation $p=1/2$. The fact that it holds for arbitrary $p\in(0,1)$ can be disturbing at one sight but it is due to the general fact that the position of a one-dimensional random walk without mean (undefined or infinite) is primarily given by the last big jumps. For more details one can consult \cite{Kesten,kesten:prob,kesten:sol}.

The proof of the following lemma is postponed to the end of this part.

\begin{lem}[randomized random walk]\label{lemme} In the setting of Theorem \ref{RW} the random walks $M$ and
$M^{\scr \xi}$ are of the same type.
\end{lem}

Therefore, in order to obtain the oscillating or drifting property of $M$ we can apply the criterion of Erickson to $M^{\scr \xi}$. It is then not difficult to see that the criterion consists of determining the convergence or divergence of the more tractable series (compare to (\ref{erick1})) given, for any  $\ell_{1},\ell_{2}$ in $\{\ttu,\ttd\}$, by
\begin{equation} \label{jgen0}
J_{\ell_{1}\mid\ell_{2}}:=\sum_{n=1}^{\infty}  \frac{n\mathbb P(\tau^{\ell_{1}}=n)}{\sum_{k=1}^{n} \mathbb P(\tau^{\ell_{2}}\geq k)}=\sum_{n=1}^{\infty}  \frac{n(-\Delta \tail_{\ell_{1}}(n))}{\sum_{k=1}^{n} \tail_{\ell_{2}}(k)},
\end{equation}
where $\Delta V(n)$ denotes the forward discrete derivative at point $n$ of the real sequence $(V_{n})$, {\it i.e.}
\begin{equation}
\Delta V(n)=V(n+1)-V(n).
\end{equation}
As stated in the theorem \ref{undefinedrt} below, the criterion can be rewritten with the series defined  by
\begin{equation} \label{kgen}
K_{\ell_{1}\mid\ell_{2}}:=\sum_{n=1}^{\infty}\left (1-\frac{n \tail_{\ell_2}(n)}{\sum_{k=1}^n \tail_{\ell_2}(k)} \right ) \frac{\mathcal T_{\ell_{1}}(n)}{\sum_{k=1}^{n} \mathcal T_{\ell_{2}}(k)}.
\end{equation}
Compare to $J_{\ell_{1}\mid\ell_{2}}$, the quantities $K_{\ell_{1}\mid\ell_{2}}$ have the advantage to involve only the distribution tails and not their derivatives, \emph{i.e.} their densities. The distribution tails are obviously more tractable in computations because of their monotonicity. It turns out the quantity
\begin{equation} \label{error-term}
1-\frac{n \tail_\ell(n)}{\sum_{k=1}^n \tail_\ell(k)}
\end{equation}
may be arbitrarily small, for instance when the distribution corresponding to $\tail_\ell$ is slowly varying. We refer to \cite{BGT} for further considerations about regularly and slowly varying functions. However, as soon as this quantity is well-controlled, typically when it stays away from $0$,   the criterion can be rewritten in terms of
\begin{equation}
\widetilde K_{\ell_1\mid\ell_2}:=\sum_{n=1}^\infty \frac{\tail_{\ell_1}(n)}{\sum_{k=1}^n \tail_{\ell_2}(k)}.
\end{equation}
Finally, it is worth noting that the quantity in \eqref{error-term} remains non-negative as shown in the proof of Theorem \ref{undefinedrt} below. We want also to point out that the {S}tolz-{C}esar\'o lemma \cite{StolCes,Stoltz2011}, a discrete L'Hôpital's rule, can be convenient to study the asymptotic of the latter quantities.

\begin{theo}[recurrence criterium and \emph{a.s.\@} asymptotics]\label{undefinedrt}
\label{cas-sans-drift}
The persistent random walk $S$ is recurrent if and only if
\begin{equation}
\label{J-infini}
J_{\ttu\mid\ttd}=\infty\quad \mbox{and}\quad J_{\ttd\mid\ttu}=\infty,
\end{equation}
and transient to $\infty$ (resp. transient to $-\infty$) if and only if
\begin{equation}
\label{un-J-fini}
J_{\ttu\mid\ttd}=\infty\quad\mbox{and}\quad J_{\ttd\mid\ttu}<\infty\quad(\mbox{resp. }\quad J_{\ttu\mid \ttd}<\infty\quad\mbox{and}\quad J_{\ttd\mid\ttu}=\infty).
\end{equation}
Moreover, when $J_{\ttu\mid\ttd}=\infty$ (resp. $J_{\ttd\mid\ttu}=\infty$),
\begin{equation}\label{limsup}
\limsup_{t\to\infty}\frac{S_{n}}{n}=1 \quad \quad \left ( \mbox{resp.}\quad \liminf_{t\to\infty}\frac{S_{n}}{n}=-1 \right ) \quad a.s..
\end{equation}
Alternatively, the quantities $J_{\ttu|\ttd}$ and $J_{\ttd|\ttu}$ can be substituted with ${K}_{\ttu|\ttd}$ and $K_{\ttd|\ttu}$ respectively.
\end{theo}

Note that the case in which $J_{\ttu\mid\ttd}$ and $J_{\ttd\mid\ttu}$ are both finite does not appear in the theorem. In fact, it follows from \cite{Erickson2} that, in such a case, the drift of the persistent random walk $S$ is well-defined and belongs to $(-1,1)$ which is excluded from this section. This theorem ends the characterization of the type of persistent random walks. In Table \ref{tableau} the conditions for the recurrence and the transience are summarized and we give some applications of these criteria below.

\begin{table}[!ht]
\renewcommand\arraystretch{1.2}
\setlength{\tabcolsep}{.2cm}
\begin{center}
\begin{tabular}{|c|c|c|c|c|}
\hline
 & \multicolumn{2}{c|}{$\Theta_{\ttu}(\infty) < \infty$} & \multicolumn{2}{c|}{$\Theta_{\ttu}(\infty) = \infty$} \\
\hline
\multirow{4}*{$\Theta_\ttd(\infty) < \infty$} & & Transient to $+\infty$ & \multicolumn{2}{c|}{\multirow{2}*{Transient to $+\infty$}} \\
& Recurrent & $\mathbf{d}_{\scr S} \in (0,1)$ & \multicolumn{2}{c|}{\multirow{4}*{$\mathbf{d}_{\scr S}=1$}} \\
\cline{3-3}
& $\mathbf{d}_{\scr S}=0$ & Transient to $-\infty$ & \multicolumn{2}{c|}{}\\
& & $\mathbf{d}_{\scr S} \in (-1,0)$ & \multicolumn{2}{c|}{}\\

\hline
\multirow{4}*{$\Theta_\ttd(\infty) = \infty$} & \multicolumn{2}{c|}{\multirow{2}*{Transient to $-\infty$}} & & Transient to $+\infty$ \\
& \multicolumn{2}{c|}{\multirow{4}*{$\mathbf{d}_{\scr S}=-1$}} & Recurrent & $\infty = J_{\ttu\mid\ttd} > J_{\ttd\mid\ttu}$ \\
\cline{5-5}
& \multicolumn{2}{c|}{} & $J_{\ttu\mid\ttd}=J_{\ttu\mid\ttd}=\infty$ & Transient to $-\infty$ \\
& \multicolumn{2}{c|}{} & & $\infty = J_{\ttd\mid\ttu} > J_{\ttu\mid\ttd}$\\
\hline
\end{tabular}
\end{center}
\caption{\label{tableau}Recurrence and transience criteria.}
\end{table}

\subsubsection*{Example of harmonic transitions}

{\it
Consider sequences of transitions $(\alpha_{n}^{\ttu})$ and $(\alpha_{n}^{\ttd})$ in $[0,1)$ with $\lambda_\ttu$ and $\lambda_{\ttd}$ in $(0,1)$ such that for sufficiently large $n$,
\begin{equation}\label{harmonic}
\alpha_{n}^{\ttu}=\frac{\lambda_{\ttu}}{n}\quad\mbox{and}\quad  \alpha_{n}^{\ttu}=\frac{\lambda_{\ttu}}{n}.
\end{equation}
Then, using the asymptotics of tails given in (\ref{tailasymp}), the corresponding persistent random walk is recurrent if and only if $\lambda_{\ttu}=\lambda_{\ttd}$.
}

We need to stress that, for this toy example, the result does not hold if the equalities in (\ref{harmonic}) are replaced by asymptotic equivalences. These conditions are somehow very stiff on the transitions since the distribution tails of the length of runs involve an infinite product. Still, interesting perturbation criteria in various context are given in Section \ref{perturbations}.

\begin{proof}[Proof of Theorem \ref{cas-sans-drift}]
First, the statements (\ref{J-infini}) and (\ref{un-J-fini}) related to the recurrence and transience properties are direct consequences of Erickson's criteria \cite{Erickson2} and of lemma \ref{lemme} since the two-sided distribution tails of the increments of the random walk $M^{\scr \xi}$ given in (\ref{Mxi-rw}) satisfies for all $n\geq 1$,
\begin{equation*}
\mathbb P(Y^{\scr \xi}_{1}\geq n)=\mathbb P(\xi_{1}=1)\mathbb P(\tau^{\ttu}\geq n)\quad\mbox{and}\quad \mathbb P(Y^{\scr \xi}\leq -n)=	\mathbb P(\xi_{1}=0)\mathbb P(\tau^{\ttd}\geq n).
\end{equation*}

Besides, from the equalities
\begin{equation*}
T_{n}=\sum_{k=1}^{n}\tau_{k}^{\ttu}+\sum_{k=1}^{n}\tau_{k}^{\ttd}\quad\mbox{and}\quad
S_{T_{n}}=\sum_{k=1}^{n}\tau_{k}^{\ttu}-\sum_{k=1}^{n}\tau_{k}^{\ttd}
\end{equation*}
we can see that (\ref{limsup}) is satisfied if for all $c>0$,
\begin{equation}\label{bigjump}
\PP \left ( \tau_n^\ttu \geq c \sum_{k=1}^{n} \tau_k^\ttd \quad i.o.\right ) =1.
\end{equation}
In fact, using the Kolmogorov's zero–one law, we only need to prove that this probability is not zero. To this end, we can see that \cite[Theorem 5., p. 1190]{Kesten} applies  and it follows that
\begin{equation}\label{kesten}
\limsup_{n\to\infty}\frac{(Y_{n}^{\scr \xi})^{+}}{\sum_{k=1}^{n} (Y_{n}^{\scr \xi})^{-}}=\limsup_{n\to\infty}\frac{\xi_{n}\tau^{\ttu}_{n}}{\sum_{k=1}^{n}(1-\xi_{k})\tau_{k}^{\ttd}}=\infty\quad a.s..
\end{equation}
Roughly speaking, this theorem states that the position of a one-dimensional random walk with an undefined mean is essentially given by the last big jump. Introducing the counting process given for all $n\geq 1$ by
\begin{equation}\label{count}
N_{n}:=\#\{1\leq k\leq n : \xi_{k}= 0\},
\end{equation}
we shall prove that
\begin{equation}\label{randomwalk}
\left\{\sum_{k=1}^{n}(1-\xi_{k})\tau_{k}^{\ttd}\right\}_{n\geq 1}\overset{\mathcal L}{=}
\left\{\sum_{k=1}^{N_{n}}\tau_{k}^{\ttd}\right\}_{n\geq 1}.
\end{equation}
For this purpose, we will see that the sequences of increments consists of independent random variables and are equal in distribution in the following sense
\begin{equation}\label{randomwalkincr}
\left\{(1-\xi_{n})\tau^{\ttd}_{n}\right\}_{n\geq 1}\overset{\mathcal L}{=} \left\{(1-\xi_{n})\tau^{\ttd}_{N_{n}}\right\}_{n\geq 1}.
\end{equation}
First note that for any $n\geq 1$,
\begin{equation}\label{loi}
\mathbb P((1-\xi_{n})\tau_{N_{n}}^{\ttd}=0)=\mathbb P(\xi_{1}=1)=\mathbb P((1-\xi_{n})\tau_{n}^{\ttd}=0).
\end{equation}
Moreover, up to a null set, we have $\{\xi_{n}=0\}=\{N_{n}=N_{n-1}+1\}$ and $N_{n-1}$ is independent of $\xi_{n}$ and of the lengths of runs. We deduce that for any $k\geq 1$,
\begin{equation}\label{loi2}
\mathbb P((1-\xi_{n})\tau_{N_{n}}^{\ttd}=k)=\mathbb P(\xi_{1}=0,\tau_{1}^{\ttd}=k)=\mathbb P((1-\xi_{n})\tau_{n}^{\ttd}=k).
\end{equation}
Hence the increments of the random walks in (\ref{randomwalk}) are identically distributed. Since the increments on left-hand side in (\ref{randomwalkincr}) are independent, it only remains to prove the independence of those on the right-hand side in \eqref{randomwalkincr} to obtain the equality in distribution in \eqref{randomwalk}. Let us fix $n\geq 1$ and set for any non-negative integers $k_{1},\cdots,k_{n}\geq 0$,
\begin{equation*}
I_{n}:=\{1\leq j\leq n : k_{j}\neq 0\}	\quad\mbox{and}\quad m_{n}:=\mathsf{card}(I_{n}).
\end{equation*}
Remark that $\ell\longmapsto m_{\ell}$ is increasing on $I_{n}$ and up to a null set,\begin{equation*}
\bigcap_{\ell\notin I_{n}}\{\xi_{\ell}=1\}\cap\bigcap_{\ell\in I_{n}}\{\xi_{\ell}=0\}\subset \{N_{n}=m_{n}\}.
\end{equation*}
Then using (\ref{loi}) and (\ref{loi2}) together with the independence properties we can see that
\begin{equation*}
\mathbb P\left(\bigcap_{j=1}^{n} \{(1-\xi_{j})\tau_{N_{j}}^{\ttd}=k_{j}\}\right)=
\mathbb P\left(\bigcap_{\ell\notin I_{n}}\{\xi_{\ell}=1\}\cap\bigcap_{\ell\in I_{n}}\{\xi_{\ell}=0,\tau_{m_{\ell}}^{\ttd}=k_{\ell}\}\right)=
\prod_{j=1}^{n} \mathbb P((1-\xi_{j})\tau_{N_{j}}^{\ttd}=k_{j}),
\end{equation*}
which ends the proof of (\ref{randomwalk}).

Next, by the standard LLN for \emph{i.i.d.}\@ sequences, we obtain that for any integer $q$ greater than $1/p$, with probability one, the events $\{N_{n}\geq \lfloor n/q\rfloor\}$ hold for all sufficiently large $n$. We deduce by (\ref{randomwalk}) and (\ref{kesten}) that
\begin{equation*}\label{bigjump2}
\PP \left(\tau_n^\ttu \geq c \sum_{k=1}^{\left\lfloor n/q\right\rfloor} \tau_k^\ttd \quad i.o. \right)=1.
\end{equation*}
As a consequence,
\begin{equation}\label{bigjump3}
\mathbb P\left(\bigcup_{\ell=0}^{q-1}
\left\{\tau_{qn+\ell}^{\ttu}\geq c \sum_{k=1}^{n} \tau_k^\ttd\right\}\quad i.o.\right)=1.
\end{equation}
Again, applying the Kolmogorov's zero-one law, we get that the $q$ sequences of events (having the same distribution) in the latter equation occur infinitely often with probability one. We deduce that (\ref{bigjump}) is satisfied  and this achieves the proof of (\ref{limsup}).

For the alternative form of the theorem, it remains to prove that $J_{\ttu|\ttd}=\infty$ if and only if $K_{\ttu\mid\ttd}=\infty$.
Summing by parts (the so called \emph{Abel transformation}) we can write for any $r\geq 1$,
\begin{equation}
\sum_{n=1}^r \frac{n (-\Delta \tail_{\ttu}(n))}{\sum_{k=1}^n \tail_\ttd(k)} = \left[1-\frac{(r+1)\tail_{\ttu}(r+1)}{\sum_{k=1}^{r+1} \tail_\ttd(k)}\right]+\sum_{n=1}^{r} \Delta\left( \frac{n}{\sum_{k=1}^n \tail_\ttd(k)}\right)\tail_{\ttu}(n+1).
\end{equation}
Besides, a simple computation gives
\begin{equation}\label{positif}
\Delta \left (\frac{n}{\sum_{k=1}^n \tail_\ttd(k)} \right ) = \frac{\sum_{k=1}^{n} \tail_\ttd(k)-n \tail_\ttd(n+1)}{\sum_{k=1}^{n+1}\tail_\ttd(k)
\sum_{k=1}^{n\phantom{+1}}\tail_\ttd(k)}
=\frac{\mathbb E[\tau^{\ttd}\mathds 1_{\tau^{\ttd}\leq n}]}{\sum_{k=1}^{n+1}\tail_\ttd(k)
\sum_{k=1}^{n\phantom{+1}}\tail_\ttd(k)}\geq 0.
\end{equation}
It follows that
\begin{equation}\label{eqJK1}
\sum_{n=1}^r \frac{n (-\Delta \tail_{\ttu}(n))}{\sum_{k=1}^n \tail_\ttd(k)} = \left[1-\frac{(r+1)\tail_{\ttu}(r+1)}{\sum_{k=1}^{r+1} \tail_\ttd(k)}\right]+\sum_{n=1}^{r} \left(1-\frac{n\tail_{\ttd}(n+1)}{\sum_{k=1}^{n}\tail_\ttd(k)}\right)\frac{\tail_{\ttu}(n+1)}{\sum_{k=1}^{n+1}\tail_\ttd(k)}.
\end{equation}
Due to the non-negativeness of the forward discrete time derivative (\ref{positif}) and due to our assumption for $\Theta_{\ttu}(\infty)$ and $\Theta_{\ttd}(\infty)$ to be infinite, the general term of the series in the right-hand side of the latter equation is non-negative and (up to a shift) equivalent to that of $K_{\ttu|\ttd}$. Moreover, we get again from (\ref{positif}) that
\begin{equation}\label{eqJK2}
\frac{r \tail_\ttu(r)}{\sum_{k=1}^r \tail_\ttd(k)}= \sum_{m=r+1}^\infty \frac{r(-\Delta \tail_\ttu(m))}{\sum_{k=1}^r \tail_\ttd(k)} \leq \sum_{m=r+1}^\infty \frac{m (-\Delta \tail_\ttu(m))}{\sum_{k=1}^m \tail_\ttd(k)}.
\end{equation}
Thus, if $J_{\ttu|\ttd}$ is infinite then so is $K_{\ttu\mid\ttd}$. Conversely, the finiteness of $J_{\ttu\mid\ttd}$ together with the estimate \eqref{eqJK2} implies the first term on the right-hand side in \eqref{eqJK1} remains bounded achieving the proof.
\end{proof}

\begin{proof}[Proof Lemma \ref{lemme}]
Deeply exploiting the Theorem \ref{RW} stating that any non-constant random walk is, with probability one, either (trichotomy) oscillating or drifting to $\pm\infty$, the proof is organized as follows:
\begin{enumerate}
\item[1)] At first, we shall prove the result in the symmetric case $p=1/2$.
\item[2)] Secondly, we shall deduce the statement for any arbitrary $p\in(0,1)$ from the latter particular case.
\end{enumerate}

 To this end, assume that the supremum limit of $M^{\scr \xi}$ is {\it a.s.\@} infinite. Following exactly the same lines as in the proof of \eqref{bigjump}, we obtain that $M$ is non-negative infinitely often with probability one. Applying  Theorem \ref{RW} we deduce that the supremum limit of $M$ is also {\it a.s.\@} infinite. Thereafter, again from the latter theorem and by symmetry, we only need to prove that if $M^{\scr \xi}$ is drifting, then so is $M$. When the {\it i.i.d.\@} Bernoulli random variables $(\xi_{n})$ are symmetric, that is $p=1/2$, it is a simple consequence of the equalities
\begin{equation}
M^{\scr \xi}
\overset{\mathcal L}{=}
M^{\scr 1-\xi}\quad\mbox{and}\quad M=M^{\scr
\xi}+M^{\scr 1-\xi}.
\end{equation}
At this stage it is worth noting that the lemma is proved in the symmetric situation.

When $p\neq 1/2$, we shall  prove that if the supremum limit of $M$ is {\it a.s.\@} infinite, then so is  the supremum limit of $M^{\scr \xi}$. Since the converse has already been proved at the beginning, we shall deduce our lemma. It is not difficult to see that if the supremum limit of $M$ is {\it a.s.\@} infinite, then so is for the supremum limit of the subordinated random walk $Q=\{M_{qn}\}_{n\geq 0}$ for any $q \geq 1$ (see \cite{Erickson} for instance). Moreover, given an {\it i.i.d.\@} sequence of random variables $(\epsilon_{n})$ independent of the length of runs and distributed as symmetric Bernoulli distributions, we deduce from the first point that $Q$ and $Q^{\epsilon}$ are of the same type. Therefore, the supremum limit of $Q^{\epsilon}$ is also {\it a.s.\@} infinite and mimicking the proof of (\ref{bigjump}) we get that for all $c>0$,
\begin{equation*}\label{bigjump6}
1=\mathbb P\left(\sum_{\ell=1}^{q}\tau_{q(n-1)+\ell}^{\ttu}\geq c\sum_{k=1}^{qn}\tau_{k}^{\ttd}\quad i.o.\right)\leq \mathbb P\left(\bigcup_{\ell=1}^{q}\left\{\tau_{q(n-1)+\ell}^{\ttu})\geq \frac{c}{q}\sum_{k=1}^{qn}\tau_{k}^{\ttd}\right\}\quad i.o.\right).
\end{equation*}
Since the $q$ sequences of events in the right-hand side of the latter inequality are identically distributed and independent, we get from the zero-one law that, for any integers $r,q\geq 1$, and all $c>0$,
\begin{equation}\label{bigjump7}
\mathbb P\left(\tau_{rn}^{\ttu}\geq c\sum_{k=1}^{qn}\tau_{k}^{\ttd}\quad i.o.\right)=1.
\end{equation}
Let $N$ be the renewal process associated with $(\xi_{n})$ and given in (\ref{count}). Similarly to (\ref{randomwalk}) we show that
\begin{equation}\label{randomwalk2}
\{M_{n}^{\scr \xi}\}_{n\geq 1}=\left\{\sum_{k=1}^{n-N_{n}}\tau_{k}^{\ttu}-\sum_{k=1}^{N_{n}}\tau_{k}^{\ttd}\right\}_{n\geq 1}.
\end{equation}
The key point is to write, for all $n\geq 1$ and non-zero $k_{1},\cdots, k_{n}$ in $\mathbb Z$,
\begin{equation*}
\mathbb P\left(\bigcap_{j=1}^{n} \{\xi_{j}\tau_{j-N_{j}}^{\ttu}-(1-\xi_{j})\tau_{N_{j}}^{\ttd}=k_{j}\}\right)=
\mathbb P\left({}
\bigcap_{\ell\in I_{n}^{\ttu}}\{\xi_{\ell}=1,\tau_{m_{\ell}^{\ttu}}^{\ttu}=k_{\ell}\}\cap
\bigcap_{\ell\in I_{n}^{\ttd}}\{\xi_{\ell}=0,\tau_{m_{\ell}^{\ttd}}^{\ttd}=k_{\ell}\} \right),
\end{equation*}
where $m_{n}^{\ell}$, with $\ell\in\{\ttu,\ttd\}$ and $n\geq 1$, denote respectively the cardinal of the sets
\begin{equation*}
I_{n}^{\ttu}:=\{1\leq j\leq n : k_{j}> 0\}
\quad\mbox{and}\quad
I_{n}^{\ttd}:=\{1\leq j\leq n :  k_{j}< 0\}.
\end{equation*}
Thereafter, by applying the strong law of large numbers to the counting process $N$, we get that for any integer $c$ greater than $1/p$, choose $c\geq 2$, the events $\{N_{n}\geq \lfloor n/c\rfloor\}$ occur for all large enough $n$ with probability one. We deduce from (\ref{randomwalk2}) and (\ref{bigjump7}) and  that
\begin{equation*}
\limsup_{n\to\infty} M_{n}^{\scr\xi}\geq \limsup_{n\to\infty} M_{cn}^{\scr\xi}\geq \limsup_{n\to\infty}\sum_{k=1}^{(c-1)n}\tau_{k}^{\ttu}-\sum_{k=1}^{n}\tau_{k}^{\ttd}=\infty\quad a.s..
\end{equation*}
This achieves the proof  the lemma.
\end{proof}

\begin{rem} \label{asymp-rem}
Contrary to the well-defined drift case for which a small perturbation
on the parameters of a recurrent persistent random walk leads in
general to a transient behaviour, in the case of an undefined drift the
persistent random walk may stay recurrent as long as the perturbation
remains asymptotically controlled. To put it in a nutshell, the criterion is global  in the former case and asymptotic in the latter case.
\end{rem}

In respect to this  remark, in the next section, we give some examples of perturbations exhibiting stability or instability of the recurrent and transient properties in the context of Section \ref{und-drift}, that is the undefined drift situation.

\section{Perturbations results}
\label{perturbations}

\setcounter{equation}{0}

In this part, we still assume that $\Theta_\ttu(\infty)=\Theta_\ttd(\infty)=\infty$. In this context, needing weaker conclusions for our purpose, assumptions of the comparison Lemma \ref{comp} can be suitably relaxed in order to fit with the subsequent applications which can be roughly split into two cases.

The first one corresponds to the case of bounded perturbations in the sense of the condition \eqref{bounded-per} below. In particular, under slight assumptions, it is shown that, for randomly chosen probabilities of change of directions, the recurrence or transience depends essentially on their deterministic means.

In the case of unbounded perturbations, the analysis is more subtle. In fact, we distinguish two regimes depending on whether the rough rate of the $\alpha_n^{\ell}$'s is close to the right-hand side of \eqref{ex2} (termed in the sequel the upper boundary) or to the left-hand side of \eqref{ex2moins} (the lower boundary). Here, the notion of closeness has to be understood in the sense of a large $p$ in \eqref{ex2moins} and \eqref{ex2}. In short, the perturbations needs to be thin (and actually, thinner and thinner as $p$ is increasing) in the former case, whereas they can be chosen relatively thicker in the latter one. Finally, we provide example of (possibly random) lacunar unbounded perturbations for which the persistent random walk remains recurrent with $\alpha_n^{\ell}$'s probabilities with a rate close to the lower boundary.


\subsection{Asymptotic comparison lemma and application}

Lemma \ref{asymp-comp} below is somehow an improvement of the comparison Lemma \ref{comp}. Actually, it means in the context of indefinite means for the lengths of runs that it suffices to compare the tails in \eqref{comp1} asymptotically to obtain a comparison of the infimum and supremum limits of the corresponding persistent random walks. Pointing out though the coupling inequality in \eqref{comp2} no longer holds.

\begin{lem}[asymptotic comparison] \label{asymp-comp}
Let $S$ and $\widetilde S$ be two persistent random walks whose corresponding distribution tails of the lengths of runs satisfy, for $n$ large enough,
\begin{equation}\label{queues-asymptotiques}
\tail_\ttu(n) \leq \widetilde{\tail}_\ttu(n)\quad \textrm{and} \quad \tail_\ttd(n) \geq \widetilde{\tail}_\ttd(n).
\end{equation}
Then, denoting by $K_{\ttu,\ttd}$ and $\widetilde{K}_{\ttu,\ttd}$ the associated quantities defined in \eqref{kgen},
\begin{equation}\label{compJ}
K_{\ttu\mid\ttd} = \infty\;\Longrightarrow \;\widetilde{K}_{\ttu\mid\ttd}=\infty.
\end{equation}
Equivalently (see in particular (\ref{limsup})  in Theorem \ref{undefinedrt}),
\begin{equation}\label{equicomp}
\limsup_{n\to\infty} S_{n}=\infty
\;\Longrightarrow \;
\limsup_{n\to \infty}\widetilde S_{n}=\infty
\quad\mbox{or}\quad
\limsup_{n\to\infty} \frac{S_{n}}{n}=1\;\Longrightarrow \;
\limsup \frac{\widetilde S_{n}}{n}=1 \quad a.s..
\end{equation}
Again, the quantity $K_{\ttu\mid\ttd}$ can be substituted with $J_{\ttu\mid\ttd}$ given in (\ref{jgen0}) and a similar comparison lemma can be deduced by symmetry exchanging $\ttu$ and $\ttd$.
\end{lem}

We stress that in the proof of this lemma we make use of the quantities $K_{\ell_{1},\ell_{2}}$ and not the $J_{\ell_{1},\ell_{2}}$ ones since the comparisons of the tails (\ref{queues-asymptotiques}) are required only for large $n$. Therefore, the inequalities on the associated densities (\ref{comp3}) do not necessarily hold even asymptotically.

\begin{proof}[Proof of Lemma \ref{asymp-comp}]
Let $N\geq 1$ be such that the inequalities of (\ref{queues-asymptotiques}) are satisfied for all $n\geq N$ and consider the persistent random walks $S^{\mathtt c}$ and ${\widetilde S}^{\mathtt c}$ associated with the modified distribution tails given for any $\ell\in\{\ttu,\ttd\}$ by
\begin{equation*}
\tail_\ell^{\mathtt c}(n)=
\left\{\begin{array}{lll}
\tail_\ell(n), & \textrm{when} &  n \geq N, \\
1, & \textrm{when} & n<N,  \\
\end{array}\right.
\quad\textrm{and}\quad
\widetilde{\tail}_\ell^{\mathtt c}(n)=
\left\{\begin{array}{lll}
\widetilde{\tail}_{\ell}(n) & \textrm{when} &  n\geq N, \\
1 & \textrm{when} & n<N. \\
\end{array} \right .
\end{equation*}
Due to Lemma \ref{comp}, there exists a coupling such that $S^{\mathtt c} \leq \widetilde{S}^{\mathtt c}$ a.s.. It follows from Theorem \ref{cas-sans-drift} that (\ref{compJ}) is satisfied with the quantities $K^{\mathtt c}_{\ttu\mid\ttd}$ and  $\widetilde{K}^{\mathtt c}_{\ttu\mid\ttd}$, corresponding to  the perturbed persistent random walk $S^{\mathtt c}$ and ${\widetilde S}^{\mathtt c}$, in place of those of $S$ and $\widetilde S$. To conclude the proof, it remains to show that $K^{\mathtt c}_{\ttu\mid\ttd}$ is infinite if and only if $K_{\ttu\mid\ttd}$ is infinite, and similarly for the tilde quantities.
Since $\tail_\ell^{\mathtt c}(n)$ and $\tail_\ell(n)$ only differ for finitely many $n$, it comes that
\begin{equation*}
\frac{\tail_{\ttu}^{\mathtt c}(n)}{\left (\sum_{k=1}^n \tail^{\mathtt c}_{\ttd}(k)\right)^2}
\;\underset{n\to \infty}{\sim}\;{}
\frac{\tail_{\ttu}(n)}{\left(\sum_{k=1}^n \tail_{\ttd}(k)\right)^2}.
\end{equation*}
Here we recall that the denominators in the latter equation tends to infinity since the means of the lengths of runs are both supposed infinite. Moreover, for the same reasons, we can see that
\begin{equation*}
1- \frac{n \tail_\ttd^{\mathtt c}(n)}{\sum_{k=1}^n \tail_\ttd^{\mathtt c}(k)}
=
\frac{\sum_{k=1}^n \tail_\ttd^{\mathtt c}(k) - n \tail_\ttd^{\mathtt c}(n)}{\sum_{k=1}^n \tail_\ttd^{\mathtt c}(k)}
\;\underset{n\to\infty}{\sim}\;
\frac{\sum_{k=1}^n \tail_\ttd(k) - n \tail_\ttd(n)}{\sum_{k=1}^n \tail_\ttd(k)}=
1- \frac{n \tail_\ttd(n)}{\sum_{k=1}^n \tail_\ttd(k)}.
\end{equation*}
Indeed, the numerators around the equivalent symbol are nothing but truncated means of the lengths of runs and go to infinity. Consequently, the proof follows from the two latter equations and the expression given in (\ref{kgen}).
\end{proof}

Now we apply this lemma to the study of the recurrence and the transience of persistent random walks when the associated distribution tails are of the same order. As the previous Remark \ref{asymp-rem} has already emphasized, the recurrent or the transient behaviour of a persistent random walk, provided the persistence times are not integrable, depend only on the asymptotic properties of their distribution tails. More precisely, the following property holds.

\begin{prop}[same order tails]\label{sameorder} Two persistent random walks $S$ and $\widetilde S$ are simultaneously recurrent or transient if their associated distribution tails of the lengths of runs satisfy, for any $\ell\in\{\ttu,\ttd\}$,
\begin{equation*}\label{comptail}
\tail_\ell(n) \asymp \widetilde{\tail}_\ell(n).
\end{equation*}

\end{prop}

\begin{proof}[Proof of proposition \ref{sameorder}]
From the assumptions follow that there exist positive constants $c_{\mathtt u}$ and $c_{\mathtt d}$ such that for $n$ sufficiently large,
\begin{equation*}
\tail_\ttu(n) \leq c_{\ttu} \widetilde{\tail}_\ttu(n)
\quad\textrm{and}\quad
\tail_\ttd(n) \geq c_{\ttd}^{-1} \widetilde{\tail}_\ttd(n).
\end{equation*}
Considering the following truncated distribution tails satisfying for all sufficiently large $n$,
\begin{equation*}
\widetilde \tail_\ttu^{\mathtt c}(n):=c_{\ttu}\widetilde{\tail}_\ttu(n)
\quad\textrm{and}\quad
\widetilde {\tail}_\ttd^{\mathtt c}(n):=c_{\ttd}^{-1} \widetilde{\tail}_\ttd(n).
\end{equation*}
Then according to Lemma \ref{asymp-comp}  the corresponding quantities (\ref{kgen}) satisfy
\begin{equation*}
K_{\ttu\mid\ttd} = \infty\;\Longrightarrow\; \widetilde K_{\ttu\mid\ttd}^{\mathtt c}=\infty.
\end{equation*}
Besides, one can see that the general terms of the series defining $\widetilde K_{\ttu\mid\ttd}^{\mathtt c}$ and $\widetilde K_{\ttu\mid\ttd}$ are equivalent, namely
\begin{equation*}
\left (1-\frac{n \widetilde{\tail}_{\ttd}^{\mathtt c}(n)}{\sum_{k=1}^n \widetilde{\tail}_{\ttd}^{\mathtt c}(k)} \right ) \frac{\widetilde{\mathcal T}_{\ttu}^{\mathtt c}(n)}{\sum_{k=1}^{n} \widetilde{\mathcal T}_{\ttd}^{\mathtt c}(k)}\underset{n\to\infty}{\sim}
c_{\ttu}c_{\ttd}\left (1-\frac{n \widetilde{\tail}_{\ttd}(n)}{\sum_{k=1}^n \widetilde{\tail}_{\ttd}(k)} \right ) \frac{\widetilde{\mathcal T}_{\ttu}(n)}{\sum_{k=1}^{n} \widetilde{\mathcal T}_{\ttd}(k)}.
\end{equation*}
Therefore if $K_{\ttu\mid\ttd}$ is infinite, so is $\widetilde K_{\ttu\mid\ttd}$ and conversely by symmetry of the problem. The proposition follows from Theorem \ref{cas-sans-drift}.
\end{proof}


\subsection{Bounded perturbation criterion and application to random perturbations}

Let $S$ be a persistent random walk (without drift) associated with the probabilities  of change of directions still denoted by $(\alpha^{\ell}_{n})$ for $\ell\in\{\ttu,\ttd\}$. By a perturbation of this persistent random walk, we mean sequences  $(\gamma^{\ell}_n)$ satisfying, for all $n\geq 1$,
\begin{equation}\label{perturb}
\widetilde{\alpha}_{n}^{\ell}:=\alpha_{n}^{\ell}+\gamma_{n}^{\ell}\in[0,1].
\end{equation}
In addition, we say that the perturbation is bounded if for any $\ell\in\{\ttu,\ttd\}$,
\begin{equation}\label{bounded-per}
\left\{\sum_{k=1}^n \log\left(1-\frac{\gamma^{\ell}_{k}}{1-\alpha^{\ell}_{k}}\right) : n\geq 1\right\}	\quad\mbox{is bounded}.
\end{equation}
It holds if for instance
\begin{equation}\label{bounded-per2}
\limsup_{n\to\infty}\alpha_{n}^{\ell}<1, \quad
\left\{\sum_{k=1}^n \gamma_{k}^{\ell} : n\geq 1\right\}\quad\mbox{is bounded,}\quad \mbox{and}\quad \sum_{n=1}^\infty (\gamma_n^\ell)^2 < \infty.
\end{equation}

In the sequel, we shall denote by $\widetilde S$ the persistent random walk associated with the probabilities of change of directions given in (\ref{perturb}). The proposition below shows that bounded perturbations do not change the recurrence or the transience behaviour. We point out that it generally fails in the well-defined drift case even for compactly supported perturbations.

\begin{prop}[bounded perturbations]\label{bounded-per3} Under a bounded perturbation the drift remains undefined. Moreover, the original and perturbed persistent random walks $S$ and $\widetilde{S}$ are simultaneously recurrent or transient.
\end{prop}

\begin{proof}[Proof of Proposition \ref{bounded-per3}]
Simply apply Proposition \ref{sameorder} with the distribution tails associated with the persistent random walk $S$ and the perturbed one $\widetilde S$ noting that
\begin{equation*}
\log(\widetilde \tail_{\ell}(n))  = \log(\tail_{\ell}(n)) +\sum_{k=1}^n \log\left(1-\frac{\gamma^{\ell}_{k}}{1-\alpha^{\ell}_{k}}\right).
\end{equation*}
\end{proof}
This criterion leads to the next interesting example of a persistent random walk in a particular random environment. Given a probability space $(\Omega,\mathcal F,\mathbb Q)$ (the environment) we consider two sequences (not necessarily independent) of independent random variables $(A_{n}^{\ttu})$ and $(A_{n}^{\ttd})$ (not necessarily identically distributed) taking values in $[0,1)$ with $\mathbb Q$-probability one. We assume furthermore that their mean sequences $(\alpha^{\ttu}_{n})$ and $(\alpha^{\ttd}_{n})$ defined by
\begin{equation*}
\alpha^{\ttu}_{n}:=\mathbb E[A_{n}^{\ttu}]\in[0,1)\quad\mbox{and}\quad \alpha^{\ttd}_{n}:=\mathbb E[A_{n}^{\ttd}]\in[0,1),
\end{equation*}
lead to a persistent random walk $S$ with an undefined drift. Besides, we introduce the variance sequences $(v^{\ttu}_{n})$ and $(v^{\ttd}_{n})$ defined by
\begin{equation*}
v_{n}^{\ttu}:=\mathbb V[A_{n}^{\ttu}]\quad\mbox{and}\quad  v^{\ttd}_{n}:=\mathbb V[A_{n}^{\ttd}],
\end{equation*}
and for $\mathbb Q$-almost all $\omega$, we denote by $S^{\omega}$ the persistent random walk in the random medium $\omega\in\Omega$, that is the persistent random walk associated with the transitions $(A_{n}^{\ttu}(\omega))$ and $(A_{n}^{\ttd}(\omega))$.

\begin{prop}[random perturbations]\label{randomperturb}
Assume that for any $\ell\in\{\ttu,\ttd\}$,
\begin{equation} \label{condition-var}
\limsup_{n \to \infty} \alpha_n^\ell < 1 \quad \mbox{and}
\quad \sum_{n=1}^{\infty} v_{n}^{\ell} < \infty.
\end{equation}
Then, for $\mathbb Q$-almost all $\omega\in\Omega$, the drift of $S^{\omega}$ remains undefined and the latter persistent random walk is recurrent or transient simultaneously with the so called mean-persistent random walk $S$.
\end{prop}

\begin{proof}[Proof of Proposition \ref{randomperturb}]
First note that for any $\ell\in\{\ttu,\ttd\}$ and $\omega\in\Omega$, we can write
\begin{equation*}
A_n^\ell(\omega)=\alpha_n^\ell+(A_n^\ell(\omega)-\alpha_n^\ell),
\end{equation*}
so that $S^{\omega}$ can be seen as a random perturbation of $S$. Note that the random residual terms in the previous decomposition are centered and independent as $n$ varies. Moreover, the condition on the variances in \eqref{condition-var} implies by the Kolmogorov's one-series theorem (see for instance \cite [Theorem 3.10, p. 46]{varadhan:01}) that
\begin{equation}
\mathbb Q\left(\lim_{n\to\infty}\sum_{k=1}^n (A_k^\ell-\alpha_k^\ell)\in\mathbb R\quad\mbox{exists}\right)=1.
\end{equation}
Again by \eqref{condition-var} and the independence of the $A_{n}^{\ell}$'s, for any $\ell\in\{\ttu,\ttd\}$, we get that
\begin{equation}
\sum_{n=1}^\infty (A_n^\ell-\alpha_n^\ell)^2 < \infty\quad\mathbb Q-a.s.
\end{equation}
Then the proposition follows from Proposition \ref{bounded-per3} and more precisely conditions given in (\ref{bounded-per2}).
\end{proof}

In the following example, we retrieve the result on harmonic transitions given in (\ref{harmonic}).

\subsubsection*{Example of random harmonic transitions}

{\it
Let us consider for any $\ell\in\{\ttu,\ttd\}$ a sequence of independent random variables  $(\varepsilon_{n}^{\ell})$ with common means $\lambda_{\ell}\in(0,1)$, almost surely bounded, and such that for all $n\geq 1$,
\begin{equation*}
A_{n}^{\ell}:=\frac{\varepsilon^{\ell}_{n}}{n}\in[0,1)\quad a.s..
\end{equation*}
Then the persistent random walk in the random environment given by the random probabilities of change of directions $(A_{n}^{\ell})$ is almost surely recurrent or almost surely transient according as the means $\lambda_{\ttu}$ and $\lambda_{\ttd}$ are equal or not.}


\subsection{Relative thinning and thickening of the perturbations near the  boundaries}

As opposed to a bounded perturbation, a perturbation is said unbounded if (\ref{bounded-per}) is not satisfied. Even though general criteria for unbounded perturbations of persistent random walks can be investigated, these criteria are tedious to write precisely and irrelevant to have an insight into the phenomena. Nevertheless, we highlight some families of examples.

To this end, introduce for any $p\geq 0$ the so called boundaries defined for large $n$ by
\begin{equation}\label{boundaries}
\widecheck \beta^{(p)}_{n}:=\frac{1}{n\log(n)\cdots\log_{[p]}(n)}\quad\mbox{and}\quad
\quad\widehat\beta_{n}^{(p)}:=\frac{1}{n}+\frac{1}{n\log(n)}+\cdots+\frac{1}{n\log(n)\cdots\log_{[p]}(n)}.
\end{equation}

Close meaning $p$ large in \eqref{boundaries}, we show that the closer we are to the sequence given in the right-hand side of \eqref{boundaries} (the so called upper boundary) the thinner the perturbation needs to be to avoid a phase transition between recurrence and transience. Conversely, the closer we are to the sequence given in the left-hand side of \eqref{boundaries} (the so called lower boundary) the thicker perturbation can be while the recurrence or transience behaviour remains unchanged.

We start with the toy example of harmonic transitions. The proofs of the following three propositions are based on the application of Theorem \ref{cas-sans-drift} and the determination of the order of distribution tails given in (\ref{tailasymp}). Their proofs are left to the reader.

\begin{prop}[harmonic case between the boundaries]
Let us consider for any $\lambda\in(0,1)$ and $c\in\mathbb R$, probabilities of change of directions staying away from one and satisfying, for $n$ sufficiently large,
\begin{equation}\label{perturb1}
\alpha_{n}^{\ttu}:=\frac{\lambda}{n} \quad\mbox{and}\quad \alpha_{n}^{\ttd}:=\frac{\lambda}{n}+\frac{c}{n\log(n)}.
\end{equation}
Then the associated persistent random walks are recurrent or  transient according as $|c|\leq 1$ or $|c|>1$. In particular, the relative maximal perturbation while remaining recurrent is of order
\begin{equation*}
\frac{|\alpha_{n}^{\ttd}-\alpha_{n}^{\ttu}|}{\alpha_{n}^{\ttu}}=
\frac{1}{\lambda\log(n)}.
\end{equation*}
\end{prop}

When $\lambda=1$ the permitted perturbation leaving unchanged the recurrent behaviour is negligible with respect to that in (\ref{perturb1}).  More precisely, this general phenomenon appears near the upper boundary.

\begin{prop}[Thining near the upper boundary]
Let us consider for any $p\geq 0$ and $c\in\mathbb R$, probabilities of change of directions staying away from one and satisfying, for $n$ sufficiently large,
\begin{equation*}
\alpha_{n}^{\ttu}:=\widehat \beta_{n}^{(p)}\quad\mbox{and}\quad
\alpha_{n}^{\ttd}:=\widehat\beta_{n}^{(p)}+\frac{c}{n\log(n)\cdots\log_{[p+2]}(n)}.
\end{equation*}
Then the associated persistent random walks are recurrent or  transient according as $|c|\leq 1$ or $|c|>1$. In particular, the relative maximal perturbation leaving unchanged the recurrent behaviour is  of order
\begin{equation*}
\frac{|\alpha_{n}^{\ttd}-\alpha_{n}^{\ttu}|}{\alpha_{n}^{\ttu}}\leq
\frac{1}{\log(n)\cdots\log_{[p+2]}(n)}.
\end{equation*}
\end{prop}

On the contrary, close to the lower boundary, the allowed perturbations to preserve the same type of behaviour are (relatively) larger than the previous ones.

\begin{prop}[Thicking near the lower boundary]\label{underbound}
Let us consider for any $p\geq 0$  and $c\in\mathbb R$, probabilities of change of directions staying away from one and satisfying, for $n$ sufficiently large,
\begin{equation*}
\alpha_{n}^{\ttu}:=\widecheck \beta_{n}^{(p)}\quad\mbox{and}\quad
\alpha_{n}^{\ttd}:=\widecheck\beta_{n}^{(p)}+\frac{c}{n\log(n)\cdots\log_{[p+1]}(n)}.
\end{equation*}
Then the associated persistent random walks are recurrent or  transient according as $|c|\leq 1$ or $|c|>1$. In particular, the relative maximal perturbation while remaining recurrent is  of order
\begin{equation*}
\frac{|\alpha_{n}^{\ttd}-\alpha_{n}^{\ttu}|}{\alpha_{n}^{\ttu}}\leq
\frac{1}{\log_{[p+1](n)}}.
\end{equation*}
\end{prop}

We focus now on some examples of lacunar perturbations of some symmetric persistent random walks (thus recurrent). They illustrate that such a perturbed persistent random walk can remain recurrent while the lacunes have in some sense a large density in the set of integers. As previously, these examples are not as general as possible but still contain the main ideas.

\subsubsection*{Example of harmonic lacunar perturbations along prime numbers}

{\it Consider $\mathbb P\subset \mathbb N$ the set of prime number and set, for any positive integer $r$,
\begin{equation*}
\mathbb P_{r}:=\bigcup_{k=1}^{r} k\mathbb P.
\end{equation*}
Introduce for $\lambda\in(0,1]$ the transition probabilities (staying away from one) given for large $n$ by
\begin{equation*}
\alpha_{n}^{\ttu}:=\frac{\lambda}{n}\quad\mbox{and}\quad \alpha_{n}^{\ttd}:=\frac{\lambda}{n}\mathds 1_{\mathbb N\setminus \mathbb P_{r}} (n).
\end{equation*}
Then we can prove that the associated persistent random walk is recurrent if and only if
\begin{equation*}
\left(\sum_{k=1}^{r}\frac{1}{k}\right)\lambda\leq 1.
\end{equation*}
In particular, when $r=1$, it is transient when $\lambda=1$ and recurrent otherwise. The proof is a consequence of our results and the prime number theorem which imply
\begin{equation*}
\sum_{p\in\mathbb P\cap[0,n]}\frac{1}{p}\underset{n\to\infty}{=}\log\log(n)+\mathcal O(1).
\end{equation*}

}

\subsubsection*{Example of random lacunar perturbations}

{\it Fix $p\in\mathbb N$ and consider independent Bernoulli random variables $(\varepsilon_{n})$ such that for $n$ large enough,
\begin{equation*}
\mathbb P(\varepsilon_{n}=0):=\frac{1}{\log_{[p+1]}(n)}.
\end{equation*}
By the second Borel-Cantelli Lemma \cite[Theorem 5.4.11., p. 218]{durrett2010}, it comes that
\begin{equation*}
\frac{1}{2}\sum_{k=1}^{n}\mathds 1_{\{\varepsilon_{n}=0\}}\underset{n\to\infty}{\sim}\frac{1}{\log_{[p+1]}(n)}\quad a.s..
\end{equation*}
Even though the lacunar set $\mathbb L_{p}:=\{n\geq 1 : \varepsilon_n=0\}$ is {\it a.s.\@} of density zero in the set of integers, these sets are somehow {\it a.s.\@} asymptotically of density one as $p$ tends to infinity. As a matter of fact, the function in the right-hand side of the latter equation is arbitrarily slow as $p$ tends to infinity.

Furthermore, we can see by applying  Propositions \ref{underbound} and \ref{randomperturb} above that the persistent random walks associated with the random transitions (staying away from one) given for $n$ sufficiently large by
\begin{equation*}
\alpha_{n}^{\ttu}:=\widecheck \beta_{n}^{(p)}\quad\mbox{and}\quad \alpha_{n}^{\ttd}:= \widecheck \beta_{n}^{(p)}\varepsilon_{n}=\widecheck \beta_{n}^{(p)}\mathds 1_{\mathbb L_{p}}(n),
\end{equation*}
are {\it a.s.\@} recurrent.}

\subsection{Perturbed probabilized context tree with grafts}

Consider the double infinite comb defined in Figure \ref{double-peigne} and attach to each finite leaf $c\in\mathcal C$  given in (\ref{leaf}) another (possibly void) context tree $\mathbb T_{c}$ (see Figure {\ref{double-peigne-2}}). Denote by $\mathcal C_{c}$ the leaves (possibly infinite) of $\mathbb T_{c}$ such that those of the resulting context tree are given by
\begin{equation*}
\widetilde{\mathcal C}:=\bigcup_{c\in\mathcal C}\mathcal C_{c}.
\end{equation*}
Then to each leaf $\tilde c\in\widetilde{\mathcal C}$ of the modified context tree corresponds a Bernoulli distribution $q_{\tilde c}$ on $\{\ttu,\ttd\}$ so that we can define the persistent random walk $\widetilde S$ associated with this enriched probabilized context tree. We introduce the persistent random walks $\widehat S$ and $\widecheck S$ built from the double infinite comb and the respective probabilities of change of directions given by
\begin{equation*}
\widecheck \alpha_{n}^{\ttu}:= \sup\{q_{\tilde c}(\ttd) : \tilde c\in  \mathcal C_{\ttu^{n}\ttd}\}\quad\mbox{and}\quad
\widecheck \alpha_{n}^{\ttd}:= \inf\{q_{\tilde c}(\ttu) : \tilde c\in  \mathcal C_{\ttd^{n}\ttu}\},
\end{equation*}
and
\begin{equation*}
\widehat \alpha_{n}^{\ttu}:= \inf\{q_{\tilde c}(\ttd) : \tilde c\in  \mathcal C_{\ttu^{n}\ttd}\}  \quad\mbox{and}\quad \widehat \alpha_{n}^{\ttd}:= \sup\{q_{\tilde c}(\ttu) : \tilde c\in  \mathcal C_{\ttd^{n}\ttu}\}.
\end{equation*}

\begin{figure}[!ht]
\begin{center}
\includegraphics[width=0.8\linewidth]{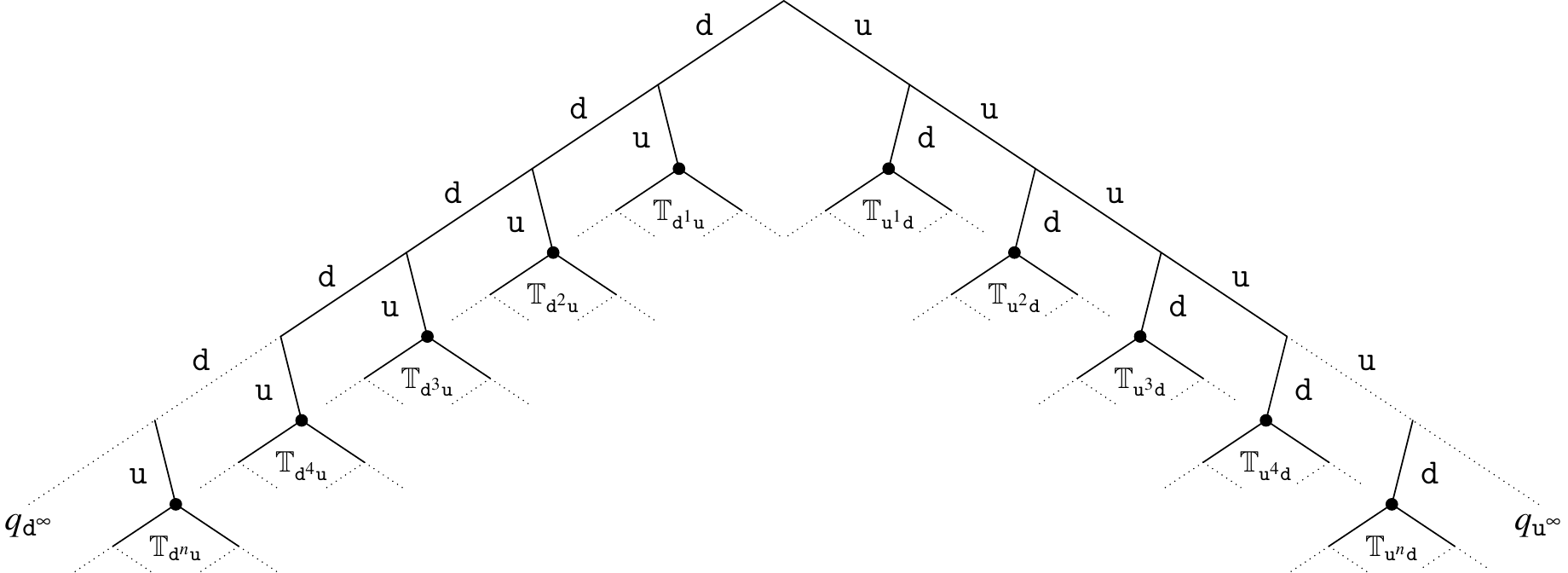}
\end{center}
\caption{\label{double-peigne-2}Probabilized context tree (grafting of the double infinite comb)}
\end{figure}

We can state the following lemma and give the following application whose proofs follow from the same coupling argument as in the proof of Lemma \ref{comp}. Remark that the situation of well-defined drift is allowed in this context.
\begin{lem}[comparison of trees]\label{trees} Assume that the persistent random walks $\widecheck S$ and $\widehat S$ are of the same recurrent or transient type, then $\widetilde S$ is recurrent or transient accordingly. More precisely, we have
\begin{equation*}
\limsup_{n\to\infty} \widecheck S_{n}=\infty\quad a.s.
\;\Longrightarrow \;
\limsup_{n\to\infty} \widetilde S_{n}=\infty\quad a.s.,
\end{equation*}
and
\begin{equation*}
\liminf_{n\to\infty} \widehat S_{n}=-\infty\quad a.s.
\;\Longrightarrow \;
\liminf_{n\to\infty} \widetilde S_{n}=-\infty\quad a.s..
\end{equation*}
\end{lem}

As an application, for instance, we can extend our results to some non-degenerated (in some sense) probabilized context trees built from a finite number of grafts from the double infinite comb.

\subsubsection*{Example of finite number of non-degenerated grafts}

{\it Consider a persistent random walk $S$ built from a double infinite comb with an undefined drift and our original assumptions. Then perturb that context tree attaching a finite number of trees $\mathbb T_{c}$, {\it i.e.\@}
\begin{equation*}
\mathsf{card}(\{c\in\mathcal C : \mathbb T_{c}\neq \emptyset\})<\infty.
\end{equation*}
Assume furthermore that the grafts satisfies, for all $n\geq 1$ such that $\mathbb T_{\ttu^{n}\ttd}\neq \emptyset$ or $\mathbb T_{\ttd^{n}\ttu}\neq \emptyset$,
\begin{equation*}
\widecheck \alpha_{n}^{\ttu}<1\quad\mbox{and}\quad {\widehat \alpha}_{n}^{\ttd}<1
\end{equation*}
Then the resulting persistent random walk $\widetilde S$ is of the same type as $S$. For example, the latter condition is satisfied when the probabilized context trees $\mathbb T_{c}$ are finite and their attached Bernoulli distributions are non-degenerated. Note that in that case, the random walk is particularly persistent in the sense that the rises and descents are no longer independent. A renewal property persists but is heavier to write.}

\bibliography{biblio-vnby}
\bibliographystyle{unsrturl}
\end{document}